\documentclass[a4paper]{amsart}
\usepackage{amssymb,verbatim,dsfont}
\usepackage{xypic}
\xyoption{all}
\usepackage{verbatim}

\def \xcirc{\objectmargin{0.1pc}\def\objectstyle{\sssize}\diagram
\squarify<1pt>{}\circled\enddiagram}

\newtheorem{theorem}{Theorem}[section]
\newtheorem{lemma}[theorem]{Lemma}
\newtheorem{proposition}[theorem]{Proposition}
\newtheorem{corollary}[theorem]{Corollary}

\theoremstyle{definition}

\newtheorem{notations}[theorem]{Notations}

\theoremstyle{remark}
\newtheorem{remark}[theorem]{Remark}

\DeclareMathOperator{\BC}{\mathsf{BC}}

\DeclareMathOperator{\BN}{\mathsf{BN}}

\DeclareMathOperator{\BP}{\mathsf{BP}}

\DeclareMathOperator{\B}{\mathsf{B}}

\DeclareMathOperator{\Ho}{\mathsf{H}}

\DeclareMathOperator{\HH}{\mathsf{HH}}

\DeclareMathOperator{\HC}{\mathsf{HC}}

\DeclareMathOperator{\HP}{\mathsf{HP}}

\DeclareMathOperator{\HN}{\mathsf{HN}}

\DeclareMathOperator{\ima}{\mathsf{Im}}

\DeclareMathOperator{\ide}{\mathsf{id}}

\DeclareMathOperator{\Tot}{\mathsf{Tot}}

\newcommand{\ot}{\otimes}
\newcommand{\sub}{\subseteq}

\newcommand{\wt}{\widetilde}
\newcommand{\wh}{\widehat}
\newcommand{\ov}{\overline}

\newcommand{\ba}{\mathbf a}
\newcommand{\bx}{\mathbf x}

\newcommand{\bb}{\mathbf b}

\newcommand{\bh}{\mathbf h}

\newcommand{\byy}{\mathbf y}
\newcommand{\bz}{\mathbf z}

\newcommand{\de}{\delta}
\newcommand{\ep}{\epsilon}

\newcommand{\si}{\sigma}

\newcommand{\cX}{\mathcal X}

\newcommand{\cY}{\mathcal Y}

\def\xcirc{\objectmargin{0.1pc}\def\objectstyle{\sssize}\diagram
\squarify<1pt>{}\circled\enddiagram}

\begin{document}

\title{Cyclic homology of crossed products}


\author{Graciela Carboni}
\address{} \curraddr{} \email{gcarbo@dm.uba.ar}
\thanks{Supported by  UBACYT 0294}

\author{Jorge A. Guccione}
\address{Departamento de Matem\'atica\\ Facultad de Ciencias Exactas y Naturales, Pabell\'on 1
- Ciudad Universitaria\\ (1428) Buenos Aires, Argentina.} \curraddr{} \email{vander@dm.uba.ar}
\thanks{Supported by UBACYT 0294 and PIP 5617 (CONICET)}

\author{Juan J. Guccione}
\address{Departamento de Matem\'atica\\ Facultad de Ciencias Exactas y Naturales\\
Pabell\'on 1 - Ciudad Universitaria\\ (1428) Buenos Aires, Argentina.} \curraddr{}
\email{jjgucci@dm.uba.ar}
\thanks{Supported by UBACYT 0294 and PIP 5617 (CONICET)}

\begin{abstract} We obtain a mixed complex, simpler that the canonical one, given the
Hochschild, cyclic, negative and periodic homology of a crossed product $E = A\#_f H$, where
$H$ is an arbitrary Hopf algebra and $f$ is a convolution invertible cocycle with values in
$A$. Actually, we work in the more general context of relative cyclic homology. Specifically,
we consider a subalgebra $K$ of $A$ which is stable under the action of $H$, and we find a
mixed complex computing the Hochschild, cyclic, negative and periodic homology of $E$ relative
to $K$. As an application we obtain two spectral sequences converging to the cyclic homology of
$E$ relative to $K$. The first one in the general setting and the second one (which generalizes
those previously found by several authors) when $f$ takes its values in $K$.\end{abstract}

\subjclass[2000]{primary 16E40;secondary 16W30}
\date{}

\dedicatory{}


\maketitle

\section*{Introduction}
Let $G$ be a group acting on a differential or algebraic manifold $M$. Then $G$ acts naturally
on the ring $A$ of regular functions of $M$, and the algebra ${}^{G\!} A$ of invariants of this
action consists of the functions that are constants on each of the orbits of $M$. This suggest
to consider ${}^{G\!} A$ as a replacement for $M/G$ in non-commutative geometry. Under suitable
conditions the invariant algebra ${}^{G\!} A$ and the smash product $A\# k[G]$, associated with
the action of $G$ on $A$, are Morita equivalent. Since $K$-theory, Hochschild homology and
cyclic homology are Morita invariant, there is no loss of information if ${}^{G\!} A$ is
replaced by $A\# k[G]$. In the general case the experience had shown that smash products are
better choices than invariants rings for algebras playing the role of noncommutative quotients.
This fact was a motivation for the interest in to develop tools to compute the cyclic homology
of smash products algebras. This problem was considered in \cite{F-T}, \cite{N} and \cite{G-J}.
In the first paper it was obtained a spectral sequence converging to the cyclic homology of the
smash product algebra $A\# k[G]$. In \cite{G-J}, this result was derived from the theory of
paracyclic modules and cylindrical modules developed by the authors. The main tool for this
computation was a version for cylindrical modules of Eilenberg-Zilber theorem. In \cite{A-K}
this theory was used to obtain a Feigin-Tsygan type spectral sequence for smash products $A\#
H$, of a Hopf algebra $H$ with an $H$-module algebra $A$.

At this point it is natural to try to extend this result to the general crossed products $A\#_f
H$ introduced in \cite{B-C-M} and \cite{D-T}. Crossed Products, and more general algebras such
as Hopf Galois extensions, have been homologically studied in several papers (see for instance
\cite{L}, \cite{S}, \cite{G-G} and \cite{J-S}) but almost all of them deal with its Hochschild
(co)homology. In \cite{J-S} the relative to $A$ cyclic homology of a Galois $H$ extension $C/A$
was studied, and the obtained results apply to the Hopf crossed products $A\#_f H$, giving the
absolute cyclic homology when $A$ is a separable algebra. As far as we know, the unique work
dealing with the absolute cyclic homology of a crossed product $A\#_f H$, with $A$ non
separable and $f$ non trivial is \cite{K-R}. In this paper the authors  get a Feigin-Tsygan
type spectral sequence for a crossed products $A\#_f H$, under the hypothesis that $H$ is
cocommutative and $f$ takes values in $k$.

The goal of this article is to present a mixed complex $\bigl(\ov{X},\ov{d},\ov{D}\bigr)$,
simpler than the canonical one, giving the Hochschild, cyclic, negative and periodic homology
of a crossed product $E=A\#_f H$. Under the assumptions of \cite{K-R} our complex is isomorphic
to the one obtain there. Our main result is Theorem~\ref{th3.2}, in which is proved that
$\bigl(\ov{X},\ov{d},\ov{D}\bigr)$ is homotopically equivalent to the canonical normalized
mixed complex $\bigl(E\ot\ov{E}^{\ot^*},b,B\bigr)$.

Actually, we work in the more general context of relative cyclic homology. Specifically, we
consider a subalgebra $K$ of $A$ which is stable under the action of $H$, and we find a mixed
complex computing the Hochschild, cyclic, negative and periodic homology of $E$ relative to $K$
(which we simply call the Hochschild, cyclic, negative and periodic homology of the $K$-algebra
$E$). As an application we obtain two spectral sequences converging to the cyclic homology of
the $K$-algebra $E$. The first one in the general setting and the second one (which generalizes
those of \cite{A-K} and \cite{K-R}) when $f$ takes values in $K$.

Our method of proof is different of the one used in \cite{G-J}, \cite{A-K} and \cite{K-R},
being based in the results obtained in \cite{G-G} and the Perturbation Lemma.

\smallskip

The paper is organized in the following way: in Section~1 we summarize the material on mixed
complexes, perturbation lemma and Hochschild homology of Hopf crossed products necessary for
our purpose. Moreover we set up notation and terminology. For proofs we refer to \cite{C} and
\cite{G-G}. In Section~2 we obtain a mixed complex $\bigl(\wh{X},\wh{d},\wh{D}\bigr)$, more
simpler that the canonical one, giving that Hochschild, cyclic, periodic and negative homology
of the $K$-algebra $E=A\#_f H$, which works without the usual assumption that $f$ is
convolution invertible. Finally in Section~3, we show that when $f$ is convolution invertible,
then $\bigl(\wh{X},\wh{d},\wh{D}\bigr)$ is isomorphic to a simpler mixed complex
$\bigl(\ov{X},\ov{d},\ov{D}\bigr)$. Finally, as an application we derive the above mentioned
spectral sequences.

\section{Preliminaries}
In this section we fix the general terminology and notation used in the following, and give a
brief review of the background necessary for the understanding of this paper.

\smallskip

Let $k$ be a commutative ring, $A$ a $k$-algebra and $H$ a Hopf $k$-algebra. We will use the
Sweedler notation $\Delta(h) = h^{(1)} \ot h^{(2)}$, with the summation understood and
superindices instead of subindices. Recall from \cite{B-C-M} and \cite{D-T} that a {\em weak
action} of $H$ on $A$ is a bilinear map $(h,a)\mapsto a^h$, from $H\times A$ to $A$, such that
for $h\in H$, $a,b\in A$

\begin{enumerate}

\smallskip

\item $(ab)^h = a^{h^{(1)}}b^{h^{(2)}}$,

\smallskip

\item $1^h = \ep(h)1$,

\smallskip

\item $a^1 = a$.

\smallskip

\end{enumerate}

Given a weak action of $H$ on $A$ and a $k$-linear map $f\:H\ot H\to A$, we let $A\#_f H$
denote the $k$-algebra (in general non-associative and without $1$) with underlying $k$-module
$A\ot H$ and multiplication map
$$
(a\ot h)(b \ot l) =  a b^{h^{(1)}}f(h^{(2)},l^{(1)}) \ot h^{(3)}l^{(2)},
$$
for all $a,b\in A$, $h,l \in H$. The element $a\ot h$ of $A\#_f H$ will usually be written $a\#
h$ to remind us $H$ is weakly acting on $A$. The algebra $A\#_f H$ is called a {\it crossed
product} if it is associative with $1\# 1$ as identity element. It is easy to check that this
happens if and only if $f$ and the weak action satisfy the following conditions:

\begin{enumerate}

\smallskip

\item[(i)] (Normality of $f$) for all $h\in H$, we have $f(h,1) = f(1,h) = \ep(h)1_A$,

\smallskip

\item[(ii)] (Cocycle condition) for all $h,l,m\in H$, we have
$$
f\bigl(l^{(1)},m^{(1)}\bigr)^{h^{(1)}} f\bigl(h^{(2)}, l^{(2)}m^{(2)}\bigr) =
f\bigl(h^{(1)},l^{(1)}\bigr) f\bigl(h^{(2)}l^{(2)},m\bigr),
$$

\smallskip

\item[(iii)] (Twisted module condition) for all $h,l\in H$, $a\in A$ we have
$$
\bigl(a^{l^{(1)}}\bigr)^{h^{(1)}} f\bigl(h^{(2)},l^{(2)} \bigr) = f\bigl(h^{(1)},l^{(1)}
\bigr)a^{h^{(2)}l^{(2)}}.
$$

\smallskip

\end{enumerate}

Next we establish some notations that we will use trough the paper.

\begin{notations} Let $K$ be a subalgebra of $A$ and let $B = A$ or $B = E$.

\begin{enumerate}

\smallskip

\item We set $\ov{D} = B/K$ and $\ov{H} = H/k$.

\smallskip

\item We use the unadorned tensor symbol $\ot$ to denote the tensor product $\ot_K$.

\smallskip

\item We write $\ov{H}^{\ot_k^l} = \ov{H}\ot_k\cdots\ot_k \ov{H}$ ($l$ times),
$\ov{D}^{\ot^l} = \ov{D}\ot\cdots\ot \ov{D}$ ($l$ times) and $\B_l(B)= B\ot\ov{D}^{\ot^l}\ot
B$.

\smallskip

\item Given $b_0\ot\cdots\ot b_r \in B^{\ot^{r+1}}$ and $0\le i<j\le r$, we write $\bb_{ij} =
b_i\ot\cdots\ot b_j$.

\smallskip

\item Given $\bh_{ij}\in H^{\ot_k^{j-i+1}}$, we set
$$
\bh_{ij}^{(1)}\ot_k \bh_{ij}^{(2)} = h_i^{(1)}\ot_k \cdots\ot_k h_j^{(1)} \ot_k h_i^{(2)}\ot_k
\cdots\ot_k h_j^{(2)}.
$$

\smallskip

\item Given $a\in A$ and $\bh_{ij}\in H^{\ot_k^{j-i+1}}$, we write $a^{\bh_{ij}} =
\Bigl(\cdots(a^{h_j})^{h_{j-1}}\cdots\Bigr)^{h_i}$.

\smallskip

\item Given $\ba_{ij}\in A^{\ot^{j-i+1}}$ and $h\in H$, we write $\ba_{ij}^h = a_i^{h^{(1)}}
\ot \cdots\ot a_j^{(h^{(j-i+1)})}$.

\smallskip

\item The symbol $\gamma(h)$ stands for $1\# h$.

\smallskip

\item  Given $\bh_{ij}\in H^{\ot_k^{j-i+1}}$, we set
$$
\qquad\gamma(\bh_{ij}) = \gamma(h_i)\ot\cdots\ot \gamma(h_j)\quad\text{and} \quad
\ov{\gamma(\bh_{ij})} = \gamma(h_i)\ot_A\cdots\ot_A \gamma(h_j).
$$

\smallskip

\item We will denote by $\mathcal{H}$ the image of the canonical inclusion of $H$ into $A\#
H$.

\smallskip

\item Given $h_1,\dots, h_i\in H$, we will denote by $\langle h_1,\dots,h_i\rangle$ the Hopf
subalgebra of $H$ generated by $h_1,\dots, h_i$.

\end{enumerate}

\end{notations}

\subsection{A simple resolution} Let $\Upsilon$ be the family of all epimorphisms of
$E$-bimodules which split as left $E$-module maps. In this subsection we review the
construction of the $\Upsilon$-projective resolution $(X_*,d_*)$, of $E$ as an $E$-bimodule,
given in Section~1 of \cite{G-G}. We are going to modify the sign of some maps in order to
obtain expressions for the boundary maps $d_*$ and the comparison maps between $(X_*,d_*)$ and
the normalized bar resolution of $E$, simpler than those of the above mentioned paper. Let $K$
be a subalgebra of $A$, closed under the weak action of $H$ on $A$. Since we want to consider
the Cyclic homology of the $K$-algebra $E$, in the sequel $\Upsilon$ will be the family of all
epimorphisms of $E$-bimodules which split as a $(E,K)$-bimodule map.

\smallskip

For all $r,s\ge 0$, let
$$
Y_s = E\ot_A (E/A)^{\ot_{\!A}^s} \ot_A E\quad\text{and}\quad X_{rs} = E\ot_A
(E/A)^{\ot_{\!A}^s} \ot \ov{A}^{\ot^r} \ot E.
$$
Consider the diagram of $E$-bimodules and $E$-bimodule maps
$$
\xymatrix{
\vdots \dto^-{-\partial_2} \\
Y_2 \dto^-{-\partial_2}  & X_{02}\lto_-{\mu_2} & X_{12} \lto_-{d^0_{12}} & \dots
\lto_-{d^0_{22}}\\
Y_1 \dto^-{-\partial_1}  & X_{01}\lto_-{\mu_1} & X_{11} \lto_-{d^0_{11}} & \dots
\lto_-{d^0_{21}}\\
Y_0 & X_{00}\lto_-{\mu_0} & X_{10} \lto_-{d^0_{10}} & \dots \lto_-{d^0_{20}},
 }
$$
where $(Y_*,\partial_*)$ is the normalized bar resolution of the $A$-algebra $E$, introduced in
\cite{G-S}; for each $s\ge 0$, the complex $(X_{*s},d^0_{*s})$ is $(-1)^s$ times the normalized
bar resolution of the algebra inclusion $K\sub A$, tensored on the left over $A$ with $E\ot_A
(E/A)^{\ot_{\!A}^s}$, and on the right over $A$ with $E$; and for each $s\ge 0$, the map
$\mu_s$ is the canonical projection.

\smallskip

Note that $X_{rs} \simeq E\ot_k \ov{H}^{\ot_k^s} \ot \ov{A}^{\ot^r} \ot E$, where the right
action of $K$ on $E\ot_k \ov{H}^{\ot_k^s}$ is the one obtained by translation of structure
through the canonical bijection from $E\ot_k \ov{H}^{\ot_k^s}$ to $E\ot_A (E/A)^{\ot_{\!A}^s}$.
Moreover, each one of the rows of this diagram is contractible as a $(E,K)$-bimodule complex. A
contracting homotopy
$$
\sigma^0_{0s}\colon Y_s \to X_{0s} \qquad\text{and}\qquad \sigma^0_{r+1,s}\colon X_{rs} \to
X_{r+1,s},
$$
of the $s$-th row, is given by
\begin{align*}
&\sigma^0_{0s}\bigl(\ov{\gamma(\bh_{0,s+1})}\bigr) = \ov{\gamma(\bh_{0s})}\ot \gamma(h_{s+1})
\intertext{and}
&\sigma^0_{r+1,s}\bigl(\ov{\gamma(\bh_{0s})}\ot \ba_{1r}\ot a_{r+1}\gamma(h)\bigr) =
(-1)^{r+s+1}\ov{\gamma(\bh_{0s})}\ot \ba_{1,r+1}\ot \gamma(h),
\end{align*}
Let $\wt{\mu}\colon Y_0\to E$ be the multiplication map. The complex of $E$-bimodules
$$
\xymatrix{E & Y_0 \lto_-{-\wt{\mu}} & Y_1 \lto_-{-\partial_1} & Y_2 \lto_-{-\partial_2} & Y_3
\lto_-{-\partial_3} & Y_4 \lto_-{-\partial_4} & Y_5 \lto_-{-\partial_5} & \dots
\lto_-{-\partial_6}}
$$
is also contractible as a complex of $(E,K)$-bimodules. A chain contracting homotopy
$$
\sigma_0^{-1}\colon E \to Y_0\qquad\text{and}\qquad \sigma^{-1}_{s+1}\colon Y_s\to
Y_{s+1}\quad\text{($s\ge 0$)},
$$
is given by $\sigma^{-1}_{s+1}(\bx) = (-1)^s \bx\ot_A 1_E$.

\medskip

For $r\ge 0$ and $1\le l\le s$, we define $E$-bimodule maps $d^l_{rs}\colon X_{rs} \to
X_{r+l-1,s-l}$ recursively on $l$ and $r$, by:
$$
d^l(\bx) = \begin{cases}
\sigma^0\xcirc\partial\xcirc\mu(\bx) &\text{if $l=1$ and $r=0$,}\\
- \sigma^0\xcirc d^1\xcirc d^0(\bx) &\text{if $l=1$ and $r>0$,}\\
- \sum_{j=1}^{l-1} \sigma^0\xcirc d^{l-j}\xcirc d^j(\bx) & \text{if $1<l$ and $r=0$,}\\
- \sum_{j=0}^{l-1} \sigma^0\xcirc d^{l-j}\xcirc d^j(\bx) &\text{if $1<l$ and $r>0$,}
\end{cases}
$$
for $\bx \in E\ot_A (E/A)^{\ot_{\!A}^s} \ot \ov{A}^{\ot^r} \ot K$.

\begin{theorem}[\cite{G-G}]\label{res nuestra} There is a $\Upsilon$-projective resolution of
$E$
\begin{equation}
\xymatrix{E & X_0\lto_{-\mu} & X_1  \lto_{d_1}  &X_2 \lto_{d_2}  & X_3 \lto_{d_3}  & X_4
\lto_{d_4} & \lto_{d_5}  \dots,} \label{eq1}
\end{equation}
where $\mu\colon X_{00}\to E$ is the multiplication map,
$$
X_n = \bigoplus_{r+s=n} X_{rs}\quad\text{and}\quad d_n = \sum^n_{l=1} d^l_{0n} + \sum_{r=1}^n
\sum^{n-r}_{l=0} d^l_{r,n-r}.
$$
\end{theorem}

In order to carry out our computations we also need to give an explicit contracting homotopy of
the resolution~\eqref{eq1}. For this we define maps
$$
\sigma^l_{l,s-l}\colon Y_s \to X_{l,s-l}\quad\text{and}\quad \sigma^l_{r+l+1,s-l}\colon X_{rs}
\to X_{r+l+1,s-l}
$$
recursively on $l$, by:
$$
\sigma^l_{r+l+1,s-l} = - \sum_{i=0}^{l-1} \sigma^0 \xcirc d^{l-i} \xcirc\sigma^i\qquad
\text{($0<l\le s$ and $r\ge -1$)}.
$$

\begin{proposition}[\cite{G-G}]\label{cont nuestra} The family
$$
\ov{\sigma}_0\colon E \to X_0,\qquad \ov{\sigma}_{n+1}\colon X_n \to X_{n+1}\quad\text{($n\ge
0$)},
$$
defined by $\ov{\sigma}_0 = \sigma_{00}^0 \xcirc \sigma_0^{-1}$ and
$$
\ov{\sigma}_{n+1} = - \sum_{l=0}^{n+1} \sigma_{l,n-l+1}^l \xcirc \sigma_{n+1}^{-1} \xcirc \mu_n
+ \sum_{r=0}^n \sum_{l=0}^{n-r} \sigma_{r+l+1,n-r-l}^l\qquad\text{($n\ge 0$)},
$$
is a contracting homotopy of~\eqref{eq1}.
\end{proposition}

Let $\wt{f}\langle h_1,\dots,h_i\rangle$ be the minimal $K$-subbimodule of $A$ including
$f\langle h_1,\dots,h_i\rangle$ and closed under the weak action of $\langle h_1,\dots,
h_i\rangle$ on $A$.

\begin{theorem}[\cite{G-G}]\label{Formula d^1} Let $\bx = \ov{\gamma(\bh_{0s})}\ot \ba_{1r}\ot
1$. The following assertions hold:
\begin{align*}
\qquad d^1(\bx) &= \sum_{i=0}^{s-1} (-1)^i
\ov{\gamma(\bh_{0,i-1})}\ot_A\gamma(h_i)\gamma(h_{i+1})\ot_A \ov{\gamma(\bh_{i+1,s})} \ot
\ba_{1r}\ot 1\\
& + (-1)^s \ov{\gamma(\bh_{0,s-1})}\ot \ba_{1r}^{h_s^{(1)}}\ot \gamma(h_s^{(2)})
\intertext{and}
d^2(\bx) & =  (-1)^{s-1} \ov{\gamma(\bh_{0,s-2})}\ot f(h_{s-1}^{(1)},h_s^{(1)}) \ov{*}\ba_{1r}
\ot \gamma(h_{s-1}^{(2)}h_s^{(2)}),
\end{align*}
where $f(h,l)\ov{*}\ba_{1r} = \sum_{i=0}^r (-1)^i (\ba_{1i}^{l^{(1)}})^{h^{(1)}}\ot
f(h^{(2)},l^{(2)})\ot \ba_{i+1,r}^{h^{(3)}l^{(3)}}$. Moreover, for each $l\ge 2$, the map
$d^l_{rs}$ takes $\bx$ into the $E$-subbimodule of $X_{r+l-1,s-l}$ generated by all the simple
tensors
$$
1\ot x_1\ot_A\cdots\ot_A x_{s-l}\ot a_1\ot\cdots\ot a_{r+l-1}\ot 1
$$
with one $a_j$ in $f\langle h_1,\dots,h_s\rangle$ and $l-2$ of the others $a_j$'s in
$\wt{f}\langle h_1,\dots,h_s\rangle$.

\end{theorem}

\subsubsection{Comparison with the normalized bar resolution} Let $(\B_*(E),b'_*)$ be the
normalized bar resolution of the algebra inclusion $A\sub E$. As it is well known, the complex
$$
\xymatrix{E & \B_0(E)\lto_-{\mu}&\B_1(E) \lto_-{b'_1} &\B_2(E) \lto_-{b'_2} & \B_3(E)
\lto_-{b'_3} &\dots \lto_-{b'_4}}
$$
is contractible as a complex of $(E,K)$-bimodules, with contracting homotopy
$$
\xi_0\colon E\to \B_0(E),\qquad \xi_{n+1}\colon \B_n(E)\to \B_{n+1}(E)\quad\text{($n\ge 0$),}
$$
given by $\xi_n(\bx) = (-1)^n \bx\ot 1$. Let
$$
\phi_*\colon (X_*,d_*)\to (\B_*(E),b'_*)\qquad\text{and}\qquad \psi_*\colon (\B_*(E),b'_*)\to
(X_*,d_*)
$$
be the morphisms of $E$-bimodule complexes, recursively defined by
\begin{align*}
&\phi_0 = \ide,\quad \psi_0 = \ide,\quad \phi_{n+1}(\bx\ot 1) = \xi_{n+1}\xcirc\phi_n\xcirc
d_{n+1}(\bx\ot 1)\\
\intertext{and}
&\psi_{n+1}(\byy\ot 1) = \ov{\sigma}_{n+1}\xcirc\psi_n\xcirc b'_{n+1}(\byy\ot 1).
\end{align*}

\begin{proposition}[\cite{G-G}]\label{homotopia} $\psi\xcirc\phi = \ide$ and
$\phi\xcirc\psi$ is homotopically equivalent to the identity map. A homotopy
$\omega_{*+1}\colon \phi_*\xcirc\psi_*\to \ide_*$ is recursively defined by
$$
\omega_1 = 0\quad\text{and}\quad\omega_{n+1}(\bx) = \xi_{n+1}\xcirc(\phi_n\xcirc \psi_n - \ide
- \omega_n\xcirc b'_n)(\bx),
$$
for $\bx\in E\ot\ov{E}^{\ot^n}\ot K$.
\end{proposition}

\begin{remark} Since $\omega\bigl(E\ot\ov{E}^{\ot^{n-1}}\ot K\bigr)\subseteq E\ot
\ov{E}^{\ot^n} \ot K$ and $\xi$ vanishes on $E\ot\ov{E}^{\ot^n}\ot K$,
$$
\omega(\bx_{0n}\ot 1) = \xi\bigl(\phi\xcirc \psi(\bx_{0n}\ot 1) - (-1)^n
\omega(\bx_{0n})\bigr).
$$
\end{remark}

\subsubsection{The filtrations of $(\B_*(E) ,b'_*)$ and $(X_*, d_*)$} Let
$$
F^i(X_n) = \bigoplus_{0\le s\le i} E\ot_A (E/A)^{\ot_{\!A}^s}\ot \ov{A}^{\ot^{n-s}}\ot E
$$
and let $F^i(\B_n(E))$ be the $E$-subbimodule of $\B_n(E)$ generated by the tensors
$$
1\ot x_1\ot \cdots\ot x_n\ot 1
$$
such that at least $n-i$ of the $x_j$'s belong to $A$. The normalized bar resolution $(\B_*(E)
,b'_*)$ and the resolution $(X_*, d_*)$ are filtered by
\begin{align*}
& F^0(\B_*(E))\subseteq F^1(\B_*(E))\subseteq F^2(\B_*(E))\subseteq\dots
\intertext{and}
& F^0(X_*)\subseteq F^1(X_*)\subseteq F^2(X_*)\subseteq\dots,
\end{align*}
respectively. In \cite[Proposition 1.2.2]{G-G} it was proven that the maps  $\phi_*$, $\psi_*$
and $\omega_{*+1}$ preserve filtrations. In Appendix~A we are going to improve this result.

\subsection{Mixed complexes}
In this subsection we recall briefly the notion of mixed complex. For more details about this
concept we refer to \cite{Ka} and \cite{B}.

\smallskip

A mixed complex $(X,b,B)$ is a graded $k$-module $(X_n)_{n\ge 0}$, endowed with morphisms
$b\colon X_n\to X_{n-1}$ and $B\colon X_n\to X_{n+1}$, such that
$$
b\xcirc b = 0,\quad B\xcirc B = 0\quad\text{and}\quad B \xcirc b + b\xcirc B = 0.
$$
A morphism of mixed complexes $f\colon (X,b,B)\to (Y,d,D)$ is a family of maps $f\colon X_n\to
Y_n$, such that $d\xcirc f = f\xcirc b$ and $D\xcirc f= f\xcirc B$. Let $u$ be a degree~$2$
variable. A mixed complex $\cX = (X,b,B)$ determines a double complex
\[
\xymatrix{\\\\\\ \BP(\cX)=}\qquad
\xymatrix{
& \vdots \dto^-{b} &\vdots \dto^-{b}& \vdots \dto^-{b}& \vdots \dto^-{b}\\
\dots & X_3 u^{-1} \lto_-{B}\dto^-{b} & X_2 u^0\lto_-{B}\dto^-{b} & X_1 u\lto_-{B}\dto^-{b} &
X_0 u^2\lto_-{B} \\
\dots & X_2 u^{-1}\lto_-{B}\dto^-{b} & X_1 u^0\lto_-{B}\dto^-{b} & X_0 u\lto_-{B}\\
\dots & X_1 u^{-1}\lto_-{B}\dto^-{b} & X_0 u^0 \lto_-{B}\\
\dots & X_0 u^{-1} \lto_-{B},}
\]
where $b(\bx u^i) = b(\bx)u^i$ and $B(\bx u^i) = B(\bx)u^{i-1}$. By deleting the positively
numbered columns we obtain a subcomplex $\BN(\cX)$ of $\BP(\cX)$. Let $\BN'(\cX)$ be the kernel
of the canonical surjection from $\BN(\cX)$ to $(X,b)$. The quotient double complex
$\BP(\cX)/\BN'(\cX)$ is denoted by $\BC(\cX)$. The homologies $\HC_*(\cX)$, $\HN_*(\cX)$ and
$\HP_*(\cX)$, of the total complexes of $\BC(\cX)$, $\BN(\cX)$ and $\BP(\cX)$ respectively, are
called the cyclic, negative and periodic homologies of $\cX$. The homology $\HH_*(\cX)$, of
$(X,b)$, is called the Hochschild homology of $\cX$. Finally, it is clear that a morphism
$f\colon \cX\to \cY$ of mixed complexes induces a morphism from the double complex $\BP(\cX)$
to the double complex $\BP(\cY)$.

\smallskip

As usual, given a $K$-bimodule $M$, we let $M\ot$ denote the quotient $M/[M,K]$, where $[M,K]$
is the $k$-module generated by the commutators $m\lambda - \lambda m$, with $\lambda\in K$ and
$m\in M$. Moreover $[m]$ will be denote the class of an element $m\in M$ in $M\ot$. Let $C$ be
a $k$ algebra and $K\sub C$ a subalgebra. The normalized mixed complex of the $K$-algebra $C$
is the mixed complex $(C\ot \ov{C}^{\ot^*}\ot,b,B)$, where $b$ is the canonical Hochschild
boundary map and the Connes operator $B$ is given by
$$
B([c_0\ot\cdots\ot c_r]) = \sum_{i=0}^r (-1)^{ir} [1\ot c_i\ot\cdots\ot c_r\ot c_0\ot\cdots\ot
c_{i-1}].
$$
The cyclic, negative, periodic and Hochschild homology groups $\HC^K_*(C)$, $\HN^K_*(C)$,
$\HP^K_*(C)$ and $\HH^K_*(C)$, of the $K$-algebra $C$, are the respective homology groups of
$(C\ot\ov{C}^{\ot^*}\ot,b,B)$.

\subsection{The perturbation lemma}
Next, we recall the perturbation lemma. We give the more general version introduced in
\cite{C}.

\smallskip

A homotopy equivalence data
\begin{equation}
\xymatrix{(Y,\partial)\ar@<-1ex>[r]_-{i} & (X,d) \ar@<-1ex>[l]_-{p}}, \quad h\colon X_*\to
X_{*+1},\label{eq2}
\end{equation}
consists of the following:

\begin{enumerate}

\smallskip

\item Chain complexes $(Y,\partial)$, $(X,d)$ and quasi-isomorphisms $i$, $p$ between them,

\smallskip

\item A homotopy $h$ from $i\xcirc p$ to $\ide$.
\end{enumerate}

\smallskip

A perturbation~$\de$ of~\eqref{eq2} is a map $\de\colon X_*\to X_{*-1}$ such that $(d+\de)^2 =
0$. We call it small if $\ide - \de\xcirc h$ is invertible. In this case we write $A = (\ide -
\de\xcirc h)^{-1}\xcirc \de$ and we consider
\begin{equation}
\xymatrix{(Y,\partial^1)\ar@<-1ex>[r]_-{i^1} & (X,d+\de)\ar@<-1ex>[l]_-{p^1}}, \quad h^1\colon
X_*\to X_{*+1},\label{eq3}
\end{equation}
with
$$
\partial^1 = \partial + p\xcirc A\xcirc i,\quad i^1 = i + h\xcirc A\xcirc i,\quad
p^1 = p + p\xcirc A\xcirc h,\quad h^1 = h + h\xcirc A\xcirc h.
$$
A deformation retract is a homotopy equivalence data such that $p\xcirc i = \ide$. A
deformation retract is called special if $h\xcirc i = 0$, $p\xcirc h = 0$ and $h\xcirc h = 0$.

\smallskip

In all the cases considered in this paper the map $\de\xcirc h$ is locally nilpotent, and so
$(\ide - \de\xcirc h)^{-1} = \sum_{n=0}^{\infty} (\de\xcirc h)^n$.

\begin{theorem}[\cite{C}]\label{lema de perturbacion} If $\de$ is a small perturbation of the
homotopy equivalence data~\eqref{eq2}, then the perturbed data~\eqref{eq3} is a homotopy
equivalence. Moreover, if \eqref{eq2} is a special deformation retract, then~\eqref{eq3} is
also.\label{th2.1}
\end{theorem}

\section{A mixed complex giving the cyclic homology of a crossed product}
Recall that $\Upsilon$ is the family of all epimorphisms of $E$-bimodules which split as a
$(E,K)$-bimodule map. Since $(X_*,d_*)$ is a $\Upsilon$-projective resolution of $E$, the
Hochschild homology of the $K$-algebra $E$, is the homology of $E\ot_{E^e}(X_*,d_*)$. Write
$\wh{X}_{rs} = E\ot_A (E/A)^{\ot_{\!A}^s}\ot \ov{A}^{\ot^r}\ot$. It is easy to check that
$\wh{X}_{rs}\simeq E\ot_{E^e} X_{rs}$. Let $\wh{d}^l_{rs}\colon \wh{X}_{rs}\to
\wh{X}_{r+l-1,s-l}$ be the map induced by $\ide_E\ot_{E^e} d^l_{rs}$. Clearly $\wh{d}^0_{rs}$
is $(-1)^s$ times the boundary map of the normalized chain Hochschild complex of the
$K$-algebra $A$, with coefficients in $E\ot_A (E/A)^{\ot_{\!A}^s}$. Moreover, from
Theorem~\ref{Formula d^1}, it follows easily that
\begin{align*}
\wh{d}^1(\bx) & = \bigl[a_0\gamma(h_0)\gamma(h_1)\ot_A\ov{\gamma(\bh_{2s})} \ot\ba_{1r}\bigr]\\
& + \sum_{i=1}^{s-1} (-1)^i \bigl[a_0\gamma(h_0)\ot_A\ov{\gamma(\bh_{1,i-1})}\ot_A
\gamma(h_i)\gamma(h_{i+1})\ot_A \ov{\gamma(\bh_{i+2,s})} \ot\ba_{1r}\bigr]\\
& +(-1)^s\bigl[\gamma(h_s^{(2)})a_0\gamma(h_0) \ot_A\ov{\gamma(\bh_{1,s-1})}\ot
\ba_{1r}^{h_s^{(1)}}\bigr]
\intertext{and}
\wh{d}^2(\bx) &= (-1)^{s-1}\bigl[\gamma(h_{s-1}^{(2)}h_s^{(2)})a_0\gamma(h_0)\ot_A
\ov{\gamma(\bh_{0,s-2})} \ot f(h_{s-1}^{(1)},h_s^{(1)})\ov{*}\ba_{1r}\bigr],
\end{align*}
where $\bx = \bigl[a_0\gamma(h_0)\ot_A\ov{\gamma(\bh_{1s})}\ot \ba_{1r}\bigr]$ and
$f(h,l)\ov{*}\ba_{1r}$ is as in Theorem~\ref{Formula d^1}. With the above identifications the
complex $E\ot_{E^e}(X_*,d_*)$ becomes $(\wh{X}_*,\wh{d}_*)$, where
$$
\wh{X}_n = \bigoplus_{r+s = n} \wh{X}_{rs}\qquad\text{and}\qquad \wh{d}_n := \sum^n_{l=1}
\wh{d}^l_{0n} + \sum_{r=1}^n \sum^{n-r}_{l=0} \wh{d}^l_{r,n-r}.
$$
Let
$$
\wh{\phi}_*\colon  (\wh{X}_*,\wh{d}_*)\to (E\ot\ov{E}^{\ot^*}\ot,b_*)\qquad\text{and}\qquad
\wh{\psi}_*\colon (E\ot\ov{E}^{\ot^*}\ot,b_*)\to (\wh{X}_*,\wh{d}_*)
$$
be the morphisms of complexes induced by $\phi$ and $\psi$ respectively. By
Proposition~\ref{homotopia}, we have $\wh{\psi}\xcirc\wh{\phi} = \ide$ and
$\wh{\phi}\xcirc\wh{\psi}$ is homotopically equivalent to the identity map, being an homotopy
$\wh{\omega}_{*+1}\colon \wh{\phi}_*\xcirc\wh{\psi}_*\to \ide_*$, the family of maps
$$
\bigl(\wh{\omega}_{n+1}\colon E\ot\ov{E}^{\ot^n}\ot\to E\ot\ov{E}^{\ot^{n+1}}\ot\bigr)_{n\ge
0},
$$
induced by $\bigl(\omega_{n+1}\colon \B_n(E)\to \B_{n+1}(E)\bigr)_{n\ge 0}$.

\subsubsection{The filtrations of $(E\ot\ov{E}^{\ot^*}\ot,b_*)$ and $(\wh{X}_*,\wh{d}_*)$} Let
$$
F^i(\wh{X}_n) = \bigoplus_{0\le s\le i} \wh{X}_{n-s,s}.
$$
and let $F^i(E\ot\ov{E}^{\ot^n}\ot)$ be the $k$-submodule of $E\ot\ov{E}^{\ot^n}\ot$ generated
by the classes of the simple tensors $x_0\ot\cdots\ot x_n$ such that at least $n-i$ of the
elements $x_1,\dots,x_n$ belong to~$A$. The normalized Hochschild complex
$(E\ot\ov{E}^{\ot^*}\ot,b_*)$ and the complex $(\wh{X}_*,\wh{d}_*)$ are filtered by
\begin{align*}
& F^0(E\ot\ov{E}^{\ot^*}\ot)\subseteq F^1(E\ot\ov{E}^{\ot^*}\ot)\subseteq
F^2(E\ot\ov{E}^{\ot^*}\ot) \subseteq\dots
\intertext{and}
& F^0(\wh{X}_*)\subseteq F^1(\wh{X}_*)\subseteq F^2(\wh{X}_*)\subseteq\dots,
\end{align*}
respectively. From \cite[Proposition 1.2.2]{G-G} it follows immediately that the maps
$\wh{\phi}_*$, $\wh{\psi}_*$ and $\wh{\omega}_{*+1}$ preserve filtrations. In Appendix~A we are
going to improve this result.

\smallskip

Let $\wh{V}_n\subseteq \wh{V}'_n$ be the $k$-submodules of $E\ot\ov{E}^{\ot^n}\ot$ generated by
the simple tensors $\bx_{0n}$ such that $\#(\{j\ge 1:x_j\notin A\cup \mathcal{H}\})=0$ and
$\#(\{j\ge 1:x_j\notin A\cup \mathcal{H}\})\le 1$, respectively.

\smallskip

Let $h_1,\dots,h_i\in H$. Recall that $\wt{f}\langle h_1,\dots,h_i\rangle$ is the minimal
$K$-subbimodule of $A$ including $f\langle h_1,\dots,h_i\rangle$ and closed under the weak
action of $H$. We will denote by $\wh{C}_n(h_1,\dots,h_i)$ the $k$-submodule of $E\ot
\ov{E}^{\ot^n}\ot$ generated by the classes of all the simple tensors $x_0\ot\cdots\ot x_n$
with some $x_1,\dots,x_n$ in $\wt{f}\langle h_1,\dots,h_i\rangle$.

\begin{proposition}\label{prop2.1} The map $\wh{\phi}$ satisfies
$$
\wh{\phi}\bigl(\bigl[a_0\gamma(h_0) \ot_A\ov{\gamma(\bh_{1i})}\ot\ba_{1,n-i}\bigr]\bigr)\equiv
\bigl[a_0\gamma(h_0)\ot \gamma(\bh_{1i})\!*\ba_{1,n-i}\bigr] + \bigl[a_0\gamma(h_0)\ot_A
\bx\bigr],
$$
where $\bigl[a_0\gamma(h_0)\ot_A \bx\bigr]\in F^{i-1}(E\ot\ov{E}^{\ot^n}\ot)\cap \wh{V}_n\cap
\wh{C}_n(h_1,\dots,h_i)$.
\end{proposition}

\begin{proof} See Appendix~A.
\end{proof}

\begin{proposition}\label{prop2.2} If $\bx = [1\ot\bx_{1n}]\in \bigl(F^i(E\ot \ov{E}^{\ot^n}\ot)
\cap \wh{V}'_n\bigr)$, then
$$
\wh{\omega}(\bx) \in (K\ot \ov{E}^{\ot^{n+1}})\cap F^i(E\ot\ov{E}^{\ot^{n+1}}\ot)\cap
\wh{V}_{n+1}.
$$
\end{proposition}

\begin{proof} See Appendix~A.
\end{proof}

\begin{lemma}\label{le2.3} Let $B_*\colon E\ot\ov{E}^{\ot^*}\ot\to E\ot\ov{E}^{\ot^{*+1}}\ot$
be the Connes operator. The composition $B \xcirc \wh{\omega}\xcirc B\xcirc \wh{\phi}$ is the
zero map.
\end{lemma}

\begin{proof} Let $\bx = \bigl[a_0\gamma(h_0)\ot_A\ov{\gamma(\bh_{1i})}\ot\ba_{1,n-i}\bigr]\in
\wh{X}_{n-i,i}$. By Proposition~\ref{prop2.1},
$$
\wh{\phi}(\bx)\in F^i(E\ot \ov{E}^{\ot^n}\ot)\cap \wh{V}_n.
$$
Hence $B\xcirc \wh{\phi}(\bx)\in (K\ot \ov{E}^{\ot^{n+1}}) \cap F^{i+1}(E\ot
\ov{E}^{\ot^{n+1}}\ot) \cap \wh{V}'_{n+1}$, and so, by Proposition~\ref{prop2.2},
$$
\wh{\omega}\xcirc B\xcirc \wh{\phi}(\bx)\in (K\ot \ov{E}^{\ot^{n+1}}\ot)\cap F^{i+1}(E\ot
\ov{E}^{\ot^{n+1}}\ot) \cap \wh{V}_{n+2}\subseteq \ker{B},
$$
as desired.
\end{proof}

For each $n\ge 0$, let $\wh{D}_n\colon \wh{X}_n\to \wh{X}_{n+1}$ be the map $\wh{D} =
\wh{\psi}\xcirc B \xcirc \wh{\phi}$.

\begin{theorem}\label{th2.4} $\bigl(\wh{X},\wh{d},\wh{D}\bigr)$ is a mixed complex giving the
Hochschild, cyclic, negative and periodic homology of the $K$-algebra  $E$. Moreover we have
chain complexes maps
\begin{equation*}
\xymatrix{{}\save[]+<-39pt,0pt>\Drop{\Tot\bigl(\BP(\wh{X},\wh{d},\wh{D})\bigr)}\restore
\ar@<-1ex>[rr]_-{\wh{\Phi}} && {{}\save[]+<58pt,0pt> \Drop{\Tot\bigl(\BP(E\ot\ov{E}^{\ot^*}\ot,
b,B)\bigr)}} \restore \ar@<-1ex>[ll]_-{\wh{\Psi}}},
\end{equation*}
given by
$$
\wh{\Phi}_n(\bx u^i) = \wh{\phi}(\bx)u^i + \wh{\omega}\xcirc B\xcirc\wh{\phi}(\bx)u^{i-1}\quad
\text{and}\quad \wh{\Psi}_n (\bx u^i) = \sum_{j\ge 0} \wh{\psi}\xcirc (B\xcirc\wh{\omega})^j
(\bx) u^{i-j}.
$$
These maps satisfy $\wh{\Psi}\xcirc \wh{\Phi} = \ide$ and and $\wh{\Phi}\xcirc\wh{\Psi}$ is
homotopically equivalent to the identity map. A homotopy $\wh{\Omega}_{*+1}\colon
\wh{\Phi}_*\xcirc\wh{\Psi}_*\to \ide_*$ is given by
$$
\wh{\Omega}_{n+1}(\bx u^i) = \sum_{j\ge 0}\wh{\omega}\xcirc (B\xcirc\wh{\omega})^j(\bx)u^{i-j}.
$$
\end{theorem}

\begin{proof} For each $i\ge 0$, let
\begin{align*}
& \wh{\phi}u^i\colon \wh{X}_{n-21}u^i\to \bigl(E\ot\ov{E}^{\ot^{n-2i}}\ot\bigr)u^i,\\
& \wh{\phi}u^i\colon \bigl(E\ot\ov{E}^{\ot^{n-2i}}\ot\bigr)u^i\to \wh{X}_{n-21}u^i
\intertext{and}
& \wh{\omega}u^i\colon \bigl(E\ot\ov{E}^{\ot^{n-2i}}\ot \bigr)u^i\to \bigl(E\ot\ov{E}^
{\ot^{n+1-2i}}\ot \bigr)u^i,
\end{align*}
be the maps defined by $\wh{\phi}u^i(\bx u^i) = \wh{\phi}(\bx)u^i$, etcetera. By the comments
preceding Lemma~\ref{le2.3}, we have a special deformation retract
\[
\xymatrix{{}\save[]+<-39pt,0pt>\Drop{\Tot\bigl(\BC(\wh{X},\wh{d},0)\bigr)}\restore
\ar@<-1ex>[rr]_-{\bigoplus_{i\ge 0}\wh{\phi}u^i} && {{}\save[]+<52pt,0pt>\Drop{
\Tot\bigl(\BC(E\ot\ov{E}^{\ot^*}\ot,b,0)\bigr)}} \restore \ar@<-1ex>[ll]_-{\bigoplus_{i\ge
0}\wh{\psi}u^i}},\qquad \bigoplus_{i\ge 0}\wh{\omega}u^i.
\]
By applying the perturbation lemma to this datum endowed with the perturbation induced by $B$,
and taking into account Lemma~\ref{le2.3}, we obtain the special deformation retract
\begin{equation}
\xymatrix{{}\save[]+<-39pt,0pt>\Drop{\Tot\bigl(\BC(\wh{X},\wh{d},\wh{D})\bigr)}\restore
\ar@<-1ex>[rr]_-{\wh{\Phi}} && {{}\save[]+<52pt,0pt>\Drop{\Tot\bigl(\BC(E\ot \ov{E}^{\ot^*}\ot,
b,B)\bigr)}} \restore \ar@<-1ex>[ll]_-{\wh{\Psi}}},\qquad \wh{\Omega}.
\end{equation}
It is easy to see that $\wh{\Phi}$, $\wh{\Psi}$ and $\wh{\Omega}$ commute with the canonical
surjections
\begin{equation}
\Tot\bigl(\BC(\wh{X},\wh{d},\wh{D})\bigr)\to \Tot\bigl(\BC(\wh{X},\wh{d},\wh{D}) \bigr)[2]
\end{equation}
and
\begin{equation}
\Tot\bigl(\BC(E\ot\ov{E}^{\ot^*}\ot,b,B)\bigr)\to \Tot\bigl(\BC(E\ot \ov{E}^{\ot^*}\ot,
b,B)\bigr)[2].
\end{equation}
An standard argument, from these facts, finishes the proof.
\end{proof}

Let $h_1,\dots,h_i\in H$. In the sequel we let $\wh{J}_n(h_1,\dots,h_i)$ and
$H\wh{J}_{n+1}(h_1,\dots,h_i)$ denote the $k$-submodules of $\wh{X}_n$ generated by all the
classes of simple tensors \mbox{$\ov{x_{0s}}\ot\ba_{1,n-s}$} with $0\le s<n$ and some $a_j$ in
$f\langle h_1,\dots,h_i\rangle$, and for all the classes of simple tensors $\ov{x_{0s}}\ot
\ba_{1,n-s}$ with $0\le s<n$ and some $a_j$ in $\wt{f}\langle h_1,\dots,h_i\rangle$,
respectively.

\begin{proposition}\label{prop2.5} Let $\wh{R}_i = F^i(E\ot\ov{E}^{\ot^n}\ot)\setminus
F^{i-1}(E\ot\ov{E}^{\ot^n}\ot)$. The following equalities hold:

\begin{enumerate}

\smallskip

\item $\wh{\psi}\bigl(\bigl[a_0\gamma(h_0)\ot\gamma(\bh_{1i})\ot\ba_{i+1,n}\bigr]\bigr) =
\bigl[a_0\gamma(h_0)\ot_A \ov{\gamma(\bh_{1i})} \ot\ba_{i+1,n}\bigr]$.

\smallskip

\item If $\bx_{0n}\in \wh{R}_i\cap \wh{V}_n$ and there is $1\le j\le i$ such
that~$x_j\in A$, then~\hbox{$\wh{\psi}(\bx_{0n}) = 0$}.

\smallskip

\item If $\bx = \bigl[a_0\gamma(h_0)\ot \gamma(\bh_{1,i-1})\ot a_i\gamma(h_i)\ot \ba_{i+1,n}
\bigr]$, then
\begin{align*}
\qquad\qquad \wh{\psi}(\bx) & \equiv \bigl[a_0\gamma(h_0)\ot_A \ov{\gamma(\bh_{1,i-1})}\ot_A
a_i\gamma(h_i)\ot\ba_{i+1,n}\bigr]\\
& + \bigl[\gamma(h_i^{(2)})a_0\gamma(h_0)\ot_A \ov{\gamma(\bh_{1,i-1})}\ot a_i\ot
\ba_{i+1,n}^{h_i^{(1)}}\bigr],
\end{align*}
module $\bigoplus_{l=0}^{i-2} \bigl(\wh{X}_{n-l,l}\cap \wh{J}_n(h_1,\dots,h_i)\bigr)$.

\smallskip

\item If $\bigl[\bx = a_0\gamma(h_0)\ot \gamma(\bh_{1,j-1})\ot a_j\!h_j\ot\gamma(\bh_{j+1,i})
\ot \ba_{i+1,n} \bigr]$ with $j<i$, then
$$
\qquad\qquad \wh{\psi}(\bx)\equiv \bigl[a_0\gamma(h_0)\ot_A \ov{\gamma(\bh_{1,j-1})}\ot_A a_j
\gamma(h_j) \ot_A \ov{\gamma(\bh_{j+1,i})}\ot\ba_{i+1,n}\bigr],
$$
module $\bigoplus_{l=0}^{i-2} \bigl(\wh{X}_{n-l,l}\cap \wh{J}_n(h_1,\dots,h_i)\bigr)$.

\smallskip

\item If $\bx = \bigl[a_0\gamma(h_0)\ot \gamma(\bh_{1,i-1}) \ot \ba_{i,j-1}\ot a_j\gamma(h_j)
\ot\ba_{j+1,n}\bigr]$ with $j>i$, then
$$
\qquad\qquad \wh{\psi}(\bx) \equiv \bigl[\gamma(h_j^{(2)}) a_0\gamma(h_0)\ot_A
\ov{\gamma(\bh_{1,i-1})} \ot \ba_{ij}\ot \ba_{j+1,n}^{h_j^{(1)}}\bigr],
$$
module $\bigoplus_{l=0}^{i-2} \bigl(\wh{X}_{n-l,l}\cap \wh{J}_n(h_1,\dots,h_{i-1},h_j)\bigr)$.

\smallskip

\item If $\bx_{0n}\in \wh{R}_i\cap \wh{V}'_n$ and there exists $1\le j_1 < j_2\le n$
such that $x_{j_1}\in A$ and $x_{j_2}\in \mathcal{H}$, then $\wh{\psi}(\bx_{0n})=0$.

\end{enumerate}

\end{proposition}

\begin{proof} See Appendix~A.
\end{proof}

Let $\wh{\eta}_n\colon \wh{X}_n\to \wh{X}_{n+1}$, $\wh{t}_{H,n}\colon \wh{X}_n\to \wh{X}_n$ and
$\wh{t}_{A,n}\colon \wh{X}_{n+1}\to \wh{X}_{n+1}$ be the $k$-linear maps defined by
\begin{align*}
& \wh{\eta}\bigl(\bigl[a_0\gamma(h_0)\ot_A\ov{\gamma(\bh_{1i})}\ot \ba_{1,n-i}\bigr]\bigr) =
\bigl[\ov{\gamma(\bh_{0i})} \ot \ba_{1,n-i}\ot a_0\bigr],\\
& \wh{t}_H\bigl(\bigl[a_0\gamma(h_0)\ot_A \ov{\gamma(\bh_{1i})}\ot \ba_{1,n-i}\bigr]\bigr) =
\bigl[\gamma(h_i^{(2)})\ot_A a_0\gamma(h_0) \ot_A \ov{\gamma(\bh_{1,i-1})}\ot
\ba_{1,n-i}^{h_i^{(1)}}\bigr]
\intertext{and}
& \wh{t}_A\bigl(\bigl[a_0\gamma(h_0)\ot_A \ov{\gamma(\bh_{1i})} \ot \ba_{1,n-i+1}\bigr]\bigr) =
\bigl[\ov{\gamma(\bh_{0i}^{(2)})}\ot \ba_{2,n-i+1}\ot a_0a_1^{\bh_{0i}^{(1)}}\bigr],
\end{align*}
respectively

\begin{proposition}\label{prop2.6} The Connes operator $\wh{D}$ satisfies:

\begin{enumerate}

\smallskip

\item If $\bx = \bigl[a_0\ot_A \ov{\gamma(\bh_{1i})}\ot \ba_{1,n-i}\bigr]$, then
$$
\wh{D}(\bx) = \sum_{j=0}^{n-i} (-1)^{j(n-i)+n} \wh{t}_A^j\xcirc \eta(\bx),
$$
module $F^{i-1}(\wh{X}_{n+1})\cap H\wh{J}_{n+1}(h_1,\dots,h_i)$.

\smallskip

\item If $\bx = \bigl[a_0\gamma(h_0)\ot_A\ov{\gamma(\bh_{1i})}\ot \ba_{1,n-i}\bigr]$ with
$a_0\gamma(h_0) \notin A$, then
\begin{align*}
\wh{D}(\bx) & = \sum_{j=0}^i (-1)^{ji} 1\ot_A \wh{t}^j_H(\bx) + \sum_{j=0}^{n-i}
(-1)^{j(n-i)+n} \wh{t}_A^j\xcirc \eta(\bx)
\end{align*}
module $F^i(\wh{X}_{n+1})\cap H\wh{J}_{n+1}(h_1,\dots,h_i)$.
\end{enumerate}

\end{proposition}

\begin{proof} It is a direct consequence of the definition of $B$, Propositions~\ref{prop2.1}
and~\ref{prop2.5}. We leave the details to the reader.
\end{proof}

\section{The cyclic homology of a crossed product with invertible cocycle}
Let $E=A\#_f H$. Assume that the cocycle $f$ is invertible. Then, the map $\gamma$ is
convolution invertible and its inverse is given by $\gamma^{-1}(h) = f^{-1}(S(h^{(2)}),h^{(3)})
\# S(h^{(1)})$. In \cite{G-G} it was proven that under this hypothesis the complex
$(\wh{X}_*,\wh{d}_*)$ of Section~2 is isomorphic to a simpler complex $(\ov{X}_*,\ov{d}_*)$. In
this section we obtain a similar result for the mixed complex~$\bigl(\wh{X},\wh{d},
\wh{D}\bigr)$.

\smallskip

For each $r,s\ge 0$, let
$$
\ov{X}_{rs} = \Bigl(E\ot \ov{A}^{\ot^r}\ot\Bigr)\ot_k \ov{H}^{\ot_k^s}.
$$
The map $\theta_{rs}\colon\wh{X}_{rs}\to \ov{X}_{rs}$, defined by
$$
\theta_{rs}(\bx) = (-1)^{rs} \bigl[a_0\gamma(h_0) a_1\gamma(h_1^{(1)})\cdots a_s
\gamma(h_s^{(1)}) \ot \ba_{s+1,s+r}\bigr] \ot_k \bh_{1s}^{(2)},
$$
where $\bx = \bigl[a_0\gamma(h_0)\ot_A \cdots\ot_A a_s\gamma(h_s)\ot \ba_{s+1,s+r}\bigr]$, is
an isomorphism. The inverse map of $\theta_{rs}$ is the map given by
$$
\bigl[a_0\gamma(h_0)\ot \ba_{1r}\bigr]\ot_k \bh_{1s} \mapsto (-1)^{rs}
\bigl[a_0\gamma(h_0)\gamma^{-1} (h_s^{(1)}) \cdots \gamma^{-1}(h_1^{(1)}) \ot_A
\ov{\gamma(\bh_{1s}^{(2)})}\ot \ba_{1r}\bigr].
$$
Let $\ov{d}^l_{rs}\colon\ov{X}_{rs}\to \ov{X}_{r+l-1,s-l}$ be the map $\ov{d}^l_{rs} :=
\theta_{r+l-1,s-l}\circ\wh{d}^l_{rs}\circ \theta_{rs}^{-1}$. In the absolute case the following
result was obtained in \cite{G-G}. The generalization to the relative context is direct.

\begin{theorem}\label{th3.1} The Hochschild homology of the $K$-algebra $E$ is the homology of
$(\ov{X}_*,\ov{d}_*)$, where
$$
\ov{X}_n = \bigoplus_{r+s = n} \ov{X}_{rs}\qquad\text{and}\qquad \ov{d}_n := \sum^n_{l=1}
\ov{d}^l_{0n} + \sum_{r=1}^n \sum^{n-r}_{l=0} \ov{d}^l_{r,n-r}.
$$
Moreover $\ov{d}^0_{rs}$ is the boundary map of the normalized chain Hochschild complex of the
$K$-algebra $A$, with coefficients in $E$, tensored on the right over $k$ with
$\ide_{\ov{H}^{\ot^s}}$,
\begin{align*}
\ov{d}^1_{rs}(\bx)& =  (-1)^{r+s}\bigl[\gamma(h_s^{(3)})a_0\gamma(h_0)\gamma^{-1}(h_s^{(1)})\ot
\ba_{1r}^{h_s^{(2)}}\bigr] \ot_k \bh_{1,s-1} \\
& + \sum_{i=1}^{s-1}(-1)^{r+i} \bigl[a_0\gamma(h_0)\ot \ba_{1r}\ot \bh_{1,i-1}\ot h_ih_{i+1}
\bigr] \ot_k
\bh_{i+2,s}\\
& + (-1)^r \bigl[a_0\gamma(h_0)\ep(h_1)\ot \ba_{1r}\bigr] \ot_k \bh_{2s}
\intertext{and}
\ov{d}^2_{rs}(\bx) & = \sum_{i=0}^r (-1)^{i-1} \Bigl[\gamma(h_{s-1}^{(5)}h_s^{(5)})
a_0\gamma(h_0) \gamma^{-1}(h_s^{(1)}) \gamma^{-1}(h_{s-1}^{(1)})\\
&\,\ot (\ba_{1i}^{h_s^{(2)}})^{h_{s-1}^{(2)}}\ot f(h_{s-1}^{(3)},h_s^{(3)})\ot
\ba_{i+1,r}^{h_{s-1}^{(4)}h_s^{(4)}}\Bigr]\ot_k \bh_{1,s-2},
\end{align*}
where $\bx= [a_0\gamma(h_0)\ot \ba_{1r}\ot \bh_{1s}]$.
\end{theorem}

For each $n\ge 0$, let $\ov{D}_n = \theta_n\xcirc \wh{D}_n\xcirc \theta_n^{-1}$.

\begin{theorem}\label{th3.2} $\bigl(\ov{X},\ov{d},\ov{D}\bigr)$ is a mixed complex giving the
Hochschild, cyclic, negative and periodic homology of $E$. More precisely, the mixed complexes
$\bigl(\ov{X},\ov{d},\ov{D}\bigr)$ and $\bigl(E\ot \ov{E}^{\ot^*},b,B\bigr)$ are homotopically
equivalent.
\end{theorem}

\begin{proof} Clearly $\bigl(\ov{X},\ov{d}, \ov{D}\bigr)$ is a mixed complex and $\theta\colon
\bigl(\wh{X},\wh{d},\wh{D}\bigr)\to \bigl(\ov{X},\ov{d}, \ov{D}\bigr)$ is an isomorphism of
mixed complexes.  So the result follows from Theorem~\ref{th2.4}.
\end{proof}

We now are going to obtain a formula for~$\ov{D}$. To do this we need to introduce a map
$T\colon H^{\ot_k^{i+1}}\to A$ such that
$$
\gamma(h_0)\hspace{-0.5pt} \gamma^{-1}(h_i)\hspace{-0.5pt}\cdots\hspace{-0.5pt}
\gamma^{-1}(h_1) \hspace{-0.5pt} = \hspace{-0.5pt} T\bigl(h_0^{(1)}
\hspace{-0.5pt},\hspace{-0.5pt} S(h_1)^{(1)}\hspace{-0.5pt},\hspace{-0.5pt}\dots
\hspace{-0.5pt}, \hspace{-0.5pt} S(h_i)^{(1)}\bigr) \hspace{-0.5pt}\gamma
\bigl(h_0^{(2)}S(h_i)^{(2)} \hspace{-0.5pt} \cdots \hspace{-0.5pt}
S(h_1)^{(2)}\bigr)\hspace{-0.5pt}.
$$
To abbreviate notations we set
$$
\zeta = \gamma^{-1}\xcirc S^{-1}\quad\text{and}\quad U\bigl(\bh_{0i}) = T(h_0,S(h_1),
\dots,S(h_i)\bigr).
$$
Since
$$
\gamma(h_0)\gamma^{-1}(h_i)\cdots \gamma^{-1}(h_1) = \gamma(h_0)\zeta\bigl(S(h_i)\bigr)
\cdots \zeta\bigl(S(h_1)\bigr),
$$
we can solve
\begin{align*}
U(\bh_{0i}) &= \gamma(h_0^{(1)})\zeta\bigl(S(h_i)^{(1)}\bigr) \cdots\zeta
\bigl(S(h_1)^{(1)}\bigr)\gamma^{-1}\bigl(h_0^{(2)}S(h_1\cdots h_i)^{(2)}\bigr)\\
& = \gamma(h_0^{(1)})\zeta\bigl(S(h_i^{(2)})\bigr) \cdots\zeta \bigl(S(h_1^{(2)})\bigr)
\gamma^{-1} \bigl(h_0^{(2)}S(h_1^{(1)}\cdots h_i^{(1)})\bigr)\\
& = \gamma(h_0^{(1)})\gamma^{-1}(h_i^{(2)}) \cdots\gamma^{-1}(h_1^{(2)}) \gamma^{-1}
\bigl(h_0^{(2)}S(h_1^{(1)}\cdots h_i^{(1)})\bigr).
\end{align*}
We now must check that $T\bigl(h_0,S(h_1),\dots,S(h_i)\bigr)\in A$. For this it suffices to see
that this element is coinvariant under the coaction $\nu = \ide\ot \Delta$ of $A\#_f H$, which
follows easily from the fact that $\nu\bigl(\gamma^{-1}(h)\bigr) =
\gamma^{-1}\bigl(h^{(2)}\bigr)\ot S\bigl(h^{(1)}\bigr)$ and $A\#_f H$ is a comodule algebra.
Note that
$$
a_0\gamma(h_0)\gamma^{-1}(h_i)\cdots\gamma^{-1}(h_1) = a_0U\bigl(h_0^{(1)},\bh_{1i}^{(2)}
\bigr)\gamma\bigl(h_0^{(2)}S(h_1^{(1)}\cdots h_i^{(1)})\bigr).
$$

\medskip

For each $0\le i\le n$, let $F^i(\ov{X}_n) = \bigoplus_{0\le s\le i} \ov{X}_{n-s,s}$. The
complex $(\ov{X}_*,\ov{d}_*)$ is filtered by $F^0(\ov{X}_*)\subseteq F^1(\ov{X}_*) \subseteq
F^2(\ov{X}_*)\subseteq\dots$.

\smallskip

Given $h_1,\dots,h_i\in H$, we let $H\ov{J}_n(h_1,\dots,h_i)$ denote the $k$-submodule of
$\ov{X}_n$ generated by all the elements $[a_0\gamma(h_0)\ot \ba_{1r}]\ot_k \bh_{1s}$, with
$r>0$ and some $a_j\in \wt{f}(h_1,\dots,h_i)$ (for the definition of this last expression see
the discussion above Proposition~\ref{Formula d^1}).

\smallskip

Let $\ov{\eta}_n\colon \ov{X}_n\to \ov{X}_{n+1}$ and $\ov{t}_{H,n}\colon \ov{X}_{n+1}\to
\ov{X}_{n+1}$ be the $k$-linear maps defined by
\begin{align*}
\,\,\quad& \ov{\eta}(\bx) = \bigl[a_0\gamma(h_0^{(1)}) \ot\ba_{1,n-i}\bigr]\ot_k
h_0^{(2)}S(h_1^{(1)}\cdots h_i^{(1)}) \ot_k \bh_{1i}^{(2)}
\intertext{and}
& \ov{t}_H(\byy) = \bigl[\gamma(h_{i+1}^{(3)}) a_0\gamma(h_0)\gamma^{-1}(h_{i+1}^{(1)})
\ot\ba_{1,n-i}^{h_{i+1}^{(2)}}\bigr]\ot_k h_{i+1}^{(4)} \ot_k \bh_{1i},
\end{align*}
where
$$
\bx = \bigl[a_0\gamma(h_0)\ot \ba_{1,n-i}\bigr]\ot_k \bh_{1i}\quad\text{and}\quad \byy =
\bigl[a_0\gamma(h_0)\ot\ba_{1,n-i}\bigr]\ot_k \bh_{1,i+1},
$$
respectively.

\begin{theorem}\label{th3.3} If $\bx = \bigl[a_0\gamma(h_0)\ot\ba_{1,n-i}\bigr]
\ot_k\bh_{1i}$, then
\begin{align*}
\ov{D}(\bx) & = \sum_{j=0}^i (-1)^{ji+n-i} \ov{t}_H^j\xcirc\eta(\bx)\\
& + \sum_{j=0}^{n-i}(-1)^{(j+1)(n-i)}\Bigl[\gamma\bigl(h_0^{(3)}S(h_1^{(1)}\cdots
h_i^{(1)})\bigr)\gamma(h_1^{(5)})\cdots\gamma(h_i^{(5)})\\
& \,\ot \ba_{j+1,n-i} \ot a_0U(h_0^{(1)},\bh_{1i}^{(3)}) \ot \bigl(\ba_{1j}^{\bh_{1i}^{(4)}}
\bigr)^{ h_0^{(2)}S(h_1^{(2)} \cdots h_i^{(2)})}\Bigr]\ot_k \bh_{1i}^{(6)},
\end{align*}
module $F^i(\ov{X}_{n+1})\cap H\ov{J}_{n+1}(h_1,\dots,h_i)$.
\end{theorem}

\begin{proof} It follows straightforwardly from Proposition~\ref{prop2.6}, the fact that
$\ov{D} = \theta\xcirc\wh{D}\xcirc \theta^{-1}$, and the formulas of~$\theta$
and~$\theta^{-1}$.
\end{proof}

\subsection{First spectral sequence} Arguing as in \cite[Proposition 3.2]{G-G} we see that,
for each \hbox{$h\in H$}, there is a morphism of complexes
$$
\vartheta^h_*\colon (E\ot\ov{A}^{\ot^*}\ot,b_*) \to (E\ot\ov{A}^{\ot^*}\ot ,b_*),
$$
which is given by $\vartheta^h_r([a_0\gamma(h_0)\ot\ba_{1r}]) = [\gamma(h^{(3)})a_0\gamma(h_0)
\gamma^{-1}(h^{(1)})\ot\ba_{1r}^{h^{(2)}}]$ and that, for each $h,l\in H$, the endomorphisms of
$\Ho_*^K(A,E)$ induced by $\vartheta^h_*\circ \vartheta^l_*$ and by $\vartheta^{hl}_*$
coincide. So, $\Ho_*^K(A,E)$ is a left $H$-module. Let
\begin{align*}
&\wt{d}_s \colon \Ho^K_r(A,E)\ot_k \ov{H}^{\ot^s}\to  \Ho^K_r(A,E)\ot_k \ov{H}^{\ot^{s-1}}
\intertext{and}
&\wt{D}_s \colon \Ho^K_r(A,E)\ot_k \ov{H}^{\ot^s}\to \Ho^K_r(A,E)\ot_k \ov{H}^{\ot^{s+1}}
\end{align*}
be the maps induced by $\ov{d}^1_{rs}$ and $\sum_{j=0}^s (-1)^{js+r} \ov{t}_H^j\xcirc
\eta_{r+s}$, respectively.

\begin{proposition}\label{prop3.4} Assume that $\ov{H}$ is a flat $k$-module. For each $r\ge
0$,
$$
\wt{\Ho^K_r(A,E)} = \bigl(\Ho^K_r(A,E)\ot_k \ov{H}^{\ot^*},\wt{d}_*,\wt{D}_*\bigr)
$$
is a mixed complex and there is a convergent spectral sequence
$$
E^2_{sr} = \HC_s\bigl(\wt{\Ho^K_r(A,E)}\bigr) \Rightarrow \HC^K_{r+s}(E).
$$
\end{proposition}

\begin{proof} Consider the spectral sequence $(E^v_{sr},d^v_{sr})_{v\ge 0}$, associated with
the filtration
$$
F^0\bigl(\Tot(\BC(\ov{X},\ov{d},\ov{D}))\bigr) \subseteq
F^1\bigl(\Tot(\BC(\ov{X},\ov{d},\ov{D}))\bigr) \subseteq
F^2\bigl(\Tot(\BC(\ov{X},\ov{d},\ov{D}))\bigr) \subseteq \cdots
$$
of the complex $\Tot(\BC(\ov{X},\ov{d},\ov{D}))$, given by
$$
F^i\bigl(\Tot(\BC(\ov{X},\ov{d},\ov{D}))_n\bigr) = \bigoplus_{j\ge 0} F^{i-2j}(\ov{X}_{n-2j})
u^j.
$$
An straightforward computation shows that

\begin{itemize}

\smallskip

\item $E^0_{sr} = \displaystyle{\bigoplus_{j\ge 0}} \Bigl(\Bigl(E\ot\ov{A}^{\ot^r}\ot
\Bigr)\ot_k \ov{H}^{\ot^{s-2j}}\Bigr)u^j$ and $d^0_{sr}$ is  $\displaystyle{\bigoplus_{j\ge
0}}\ov{d}^0_{r,s-2j}u^j$,

\smallskip

\item $E^1_{sr} = \displaystyle{\bigoplus_{j\ge 0}} \Bigl(\Ho_r(A,E)\ot_k
\ov{H}^{\ot^{s-2j}}\Bigr)u^j$ and $d^1_{sr}$ is $\wt{d}+\wt{D}$.

\end{itemize}
From this it follows easily that $\wt{\Ho^K_r(A,E)}$ is a mixed complex and
$$
E^2_{sr} = \HC_s\bigl(\wt{\Ho^K_r(A,E)}\bigr).
$$
In order to finish the proof it suffices to note that the filtration of
$\Tot(\BC(\ov{X},\ov{d},\ov{D}))$ introduced above is canonically bounded, and so, by
Theorem~\ref{th3.2}, the spectral sequence $(E^v_{sr})_{v\ge 0}$ converges to the cyclic
homology of the $K$-algebra $E$.
\end{proof}

\begin{corollary}\label{cor3.5} If $\Ho_i^K(A,E) = 0$ for all $i>0$, then $\HC^K_n(E) =
\HC_n(\wt{\Ho^K_0(A,E)})$.
\end{corollary}

\begin{proposition}\label{prop3.6} Assume $H$ is a separable algebra and let $t$ be the
integral of $H$ satisfying $\epsilon(t) = 1$. Then
$$
E^2_{sr} = \begin{cases} \Ho_0\bigl(H,\wt{\Ho^K_r(A,E)}\bigr) &\text{if $s$ is even,}\\ 0 &
\text{if $s$ is odd,} \end{cases}
$$
and for $s$ even the map $d^2_{sr}\colon E^2_{sr}\to E^2_{s-2,r+1}$ is given by
\begin{align*}
d^2 &\Bigl(\ov{\sum [a_0\gamma(h) \ot \ba_{1r}]}\Bigr) = \sum_{j=0}^r \ov{\sum (-1)^{(j+1)r}
\Bigl[\gamma(h^{(2)})\ot a_{j+1,r}\ot a_0\ot a_{1j}^{h^{(1)}}\Bigr]}\\
& + \sum_{j=0}^r(-1)^j \ov{\sum \Bigl[\gamma(t^{(5)}h^{(4)}) a_0 \gamma^{-1}(t^{(1)})\ot
(\ba_{1j}^{h^{(1)}})^{t^{(2)}}\ot f(t^{(3)},h^{(2)})\ot \ba_{j+1,r}^{t^{(4)}h^{(3)}}\Bigr]},
\end{align*}
where $\sum [a_0\gamma(h) \ot \ba_{1r}]$ is a $r$-cycle of $\bigl(E\ot \ov{A}^{\ot^*}\ot,b_*
\bigr)$ and $\ov{\sum [a_0\gamma(h) \ot \ba_{1r}]}$ denotes its class in
$\Ho_0\bigl(H,\wt{\Ho^K_r(A,E)}\bigr)$, etcetera.
\end{proposition}

\begin{proof} The first assertion is trivial and the second one follows from a direct
computation using the construction of the spectral sequence of a filtrated complex. For this it
is convenient to note that
$$
\ov{t}_H\xcirc \eta\bigl([a_0\gamma(h)\ot\ba_{1r}]\bigr)-\ov{d}^1\bigl([a_0\gamma(h^{(1)})
\ot\ba_{1r}]\ot_k t\ot_k h^{(2)}\bigr) \in \ima(\wt{d}_s).
$$
We leave the details to the reader.
\end{proof}

\subsection{Second spectral sequence} In this subsection we assume that $f$ takes values in
$K$. Under this hypothesis the maps $\ov{d}^l$ vanish for all $l\ge 2$ and we obtain a spectral
sequence that generalizes those given in \cite{A-K} and \cite{K-R}.

\smallskip

For each $r\ge 0$, we define a map
$$
\xymatrix@R=1pt{{H\ot_k \bigl(E\ot\ov{A}^{\ot^r}\ot\bigr)}\rto & {E\ot \ov{A}^{\ot^r}}\ot,\\
{}\save[]!<-5pt,0pt>\Drop{{h\ot \bx}}\restore
\save[]!<8pt,0pt>\Drop{}\ar@{|->}[0,1]+<-22pt,0pt> \restore &
{}\save[]+<-6pt,0pt>\Drop{{h\blacktriangleright \bx}}\restore}
$$
by $h\blacktriangleright [a\gamma(u)\ot\ba_{1r}] =
\bigl[\gamma(h^{(3)})a\gamma(u)\gamma^{-1}(h^{(1)})\ot \ba_{1r}^{h^{(2)}}\bigr]$.

\begin{proposition}\label{prop3.7} For each $r\ge 0$ the map $\blacktriangleright$ is an
action of $H$ on $E\ot \ov{A}^{\ot^r}\ot$.
\end{proposition}

\begin{proof} It is trivial that $\blacktriangleright$ is unitary. Next we verify the
associative property. By definition
$$
l\blacktriangleright \bigl(h\blacktriangleright [a\gamma(u)\ot\ba_{1r}]\bigr) =
\Bigl[\gamma(l^{(3)})\gamma(h^{(3)})a\gamma(u)\gamma^{-1}(h^{(1)})\gamma^{-1}(l^{(1)})\ot
\bigl(\ba_{1r}^{h^{(2)}}\bigr)^{l^{(2)}}\Bigr].
$$
Since
$$
\bigl(\ba_{1r}^h\bigr)^l = f(l^{(1)},h^{(1)})\ba_{1r}^{l^{(2)}h^{(2)}}
f^{-1}(l^{(3)},h^{(3)}),\qquad \gamma(l)\gamma(h) = f(l^{(1)},h^{(1)}) \gamma(l^{(2)}h^{(2)})
$$
and $f^{-1}$ is the convolution inverse of $f$, we have
$$
l\blacktriangleright \bigl(h\blacktriangleright [a\gamma(u)\ot\ba_{1r}]\bigr)=
\Bigl[\gamma(l^{(4)}h^{(4)})a\gamma(u)\gamma^{-1}(h^{(1)})\gamma^{-1}(l^{(1)})\ot f(
l^{(2)},h^{(2)})\ba_{1r}^{l^{(3)}h^{(3)}}\Bigr].
$$
Using now that, by the twisted module condition applied twice,
\begin{align*}
\gamma^{-1}(h)\gamma^{-1}(l)&\!=\! f^{-1}\bigl(S(h^{(2)}), h^{(3)}\bigr) \gamma\bigl(S(h^{(1)})
\bigr)f^{-1}\bigl(S(l^{(2)}), l^{(3)}\bigr)\gamma\bigl(S(l^{(1)})\bigr)\\
&\!=\! f^{-1}\bigl(S(h^{(3)}), h^{(4)}\bigr)f^{-1}\bigl(S(l^{(3)}),
l^{(4)}\bigr)f\bigl(S(h^{(2)}),S(l^{(2)})\bigr)\gamma\bigl(S(l^{(1)}h^{(1)})\bigr)\\
&\!=\! f^{-1}\bigl(S(h^{(3)}),h^{(4)}\bigr) f^{-1}\bigl(S(h^{(2)})S(l^{(2)}),l^{(3)}\bigr)
\gamma\bigl(S(l^{(1)}h^{(1)})\bigr)\\
&\!=\! f^{-1}\bigl(S(l^{(3)}h^{(3)})l^{(4)},h^{(4)}\bigr) f^{-1}\bigl(S(l^{(2)}h^{(2)}),
l^{(5)}\bigr) \gamma\bigl(S(l^{(1)}h^{(1)})\bigr)\\
&\!=\! f^{-1}\bigl(S(l^{(2)}h^{(2)}),l^{(3)}h^{(3)}\bigr) f^{-1}\bigl(l^{(4)},h^{(4)}\bigr)
\gamma\bigl(S(l^{(1)}h^{(1)})\bigr),
\end{align*}
and again that $f^{-1}$ is the convolution inverse of $f$, we obtain
\begin{align*}
l\blacktriangleright \bigl(h\blacktriangleright[a\gamma(u)\ot\ba_{1r}]\bigr)& = \Bigl[
\gamma(v^{(5)})a\gamma(u)f^{-1}\bigl(S(v^{(2)}),v^{(3)}\bigr)\gamma\bigl(S(v^{(1)})\bigr)\ot
\ba_{1r}^{v^{(4)}}\Bigr]\\
& = \Bigl[\gamma(v^{(3)})a\gamma(u) \gamma^{-1}(v^{(1)})\ot \ba_{1r}^{v^{(2)}}\Bigr],
\end{align*}
where $v = lh$. Since the last expression equals $(lh) \blacktriangleright [a\gamma(u)\ot
\ba_{1r}]$, this finishes the proof.
\end{proof}

For each $r\ge 0$, let $\mathcal{M}_r$ be $E\ot\ov{A}^{\ot^r}\ot$, endowed with the left
$H$-module structure given by $\blacktriangleright$. For each $r,s\ge 0$, let
$\mathcal{B}_{rs}\colon \mathcal{M}_r\ot_k \ov{H}^{\ot_k^s}\to \mathcal{M}_{r+1}
\ot_k\ov{H}^{\ot_k^s}$ be the map defined by
\begin{align*}
\mathcal{B}(\bx) & = \sum_{j=0}^r(-1)^{(j+1)r}\Bigl[\gamma\bigl(h_0^{(3)}S(h_1^{(1)}
\cdots h_s^{(1)})\bigr)\gamma(h_1^{(5)})\cdots\gamma(h_s^{(5)})\\
& \,\ot \ba_{j+1,r} \ot a_0U(h_0^{(1)},\bh_{1s}^{(3)}) \ot \bigl(\ba_{1j}^{\bh_{1s}^{(4)}}
\bigr)^{ h_0^{(2)}S(h_1^{(2)} \cdots h_s^{(2)})}\Bigr]\ot_k \bh_{1s}^{(6)},
\end{align*}
where $\bx = \bigl[a_0\gamma(h_0)\ot\ba_{1r}\bigr] \ot_k\bh_{1s}$. For each $r,s\ge 0$, let
$$
\partial_r \colon \Ho_s(H,\mathcal{M}_r)\to \Ho_s(H,\mathcal{M}_{r-1})\quad\text{and}\quad
\mathcal{D}_r \colon \Ho_s(H,\mathcal{M}_r)\to \Ho_s(H,\mathcal{M}_{r+1})
$$
be the maps induced by $\ov{d}^0_{rs}$ and $\mathcal{B}_{rs}$, respectively

\begin{proposition}\label{prop3.8} For each $s\ge 0$,
$$
\wt{\Ho^K_s(H,E)} = \bigl(\Ho_s(H,\mathcal{M}_*),\partial_*,\mathcal{D}_*\bigr)
$$
is a mixed complex and there is a convergent spectral sequence
$$
\mathcal{E}^2_{rs} = \HC_r\bigl(\wt{\Ho^K_s(H,E)}\bigr) \Rightarrow \HC^K_{r+s}(E).
$$
\end{proposition}

\begin{proof} Consider the spectral sequence $(\mathcal{E}^v_{rs},\delta^v_{rs})_{v\ge 0}$,
associated with the filtration
$$
\mathcal{F}^0\bigl(\Tot(\BC(\ov{X},\ov{d},\ov{D}))\bigr) \subseteq
\mathcal{F}^1\bigl(\Tot(\BC(\ov{X},\ov{d},\ov{D}))\bigr) \subseteq
\mathcal{F}^2\bigl(\Tot(\BC(\ov{X},\ov{d},\ov{D}))\bigr) \subseteq \cdots
$$
of the complex $\Tot(\BC(\ov{X},\ov{d},\ov{D}))$, given by
$$
\mathcal{F}^i\bigl(\Tot(\BC(\ov{X},\ov{d},\ov{D}))_n\bigr) = \bigoplus_{j\ge 0}
\mathcal{F}^{i-2j}(\ov{X}_{n-2j}) u^j,
$$
where $\mathcal{F}^l(\ov{X}_m) = \bigoplus_{0\le r\le l} \ov{X}_{r,m-r}$. An straightforward
computation shows that

\begin{itemize}

\smallskip

\item $\mathcal{E}^0_{rs} = \displaystyle{\bigoplus_{j\ge 0}} \Bigl(\mathcal{M}_{r-2j}\ot_k
\ov{H}^{\ot^s}\Bigr)u^j$ and $\delta^0_{rs}$ is  $\displaystyle{\bigoplus_{j\ge 0}}
\ov{d}^1_{r-2j,s}u^j$,

\smallskip

\item $\mathcal{E}^1_{rs} = \displaystyle{\bigoplus_{j\ge 0}} \Ho_s(H,\mathcal{M}_{r-2j})u^j$
and $\delta^1_{rs}$ is $\partial+\mathcal{D}$.

\end{itemize}
From this it is easy to see that $\wt{\Ho^K_s(H,E)}$ is a mixed complex and
$$
\mathcal{E}^2_{rs} = \HC_r\bigl(\wt{\Ho^K_s(H,E)}\bigr).
$$
In order to finish the proof it suffices to note that the filtration of
$\Tot(\BC(\ov{X},\ov{d},\ov{D}))$ introduced above is canonically bounded, and so, by
Theorem~\ref{th3.2}, the spectral sequence $(\mathcal{E}^v_{rs},\delta^v_{rs})_{v\ge 0}$
converges to the cyclic homology of the $K$-algebra $E$.
\end{proof}

\begin{corollary}\label{cor3.9} If $H$ is separable, then $\HC^K_n(E) = \HC_n\bigl(\wt{\Ho^K_0
(H,E)}\bigr)$.
\end{corollary}

\section{Some decompositions of the mixed complexes}
Let $[H,H]$ be the $k$-submodule of $H$ spanned by the set of all elements $hl-lh$ with $h,l\in
H$. It is easy to see that $[H,H]$ is a coideal in $H$. Let $\breve H$ be the quotient
coalgebra $H/[H,H]$. In this section we study decompositions of the mixed complexes
$\bigl(E\ot\ov{E}^{\ot^*}\ot,b,B\bigr)$, $\bigl(\wh{X},\wh{d},\wh{D}\bigr)$ and
$\bigl(\ov{X},\ov{d},\ov{D}\bigr)$ induced by decompositions of $\breve{H}$.

For $h\in H$, we let $\ov{h}$ denote the class of $h$ in $\breve H$. Given a subcoalgebra $C$
of $\breve H$ and a right $\breve H$-comodule $(N,\rho)$, we set $N^C = \{n\in N\:\rho(n)\in
N\ot C\}$. It is well known that if $\breve H$ decomposes as a direct sum of a family
$(C_i)_{i\in I}$ of subcoalgebras, then $N= \bigoplus_{i\in I} N^{C_i}$.

\smallskip

For each $n$, the module $E\ot\ov{E}^{\ot^n}\ot$ is an $\breve H$-comodule via
$$
\rho_n\bigl(\bigl[a_0\gamma(h_0)\ot\cdots\ot a_n\gamma(h_n)\bigr]\bigr) = \Bigl[a_0
\gamma(h_0^{(1)})\ot\cdots\ot a_n\gamma(h_n^{(1)})\Bigr]\ot_k \ov{h_0^{(2)}\cdots h_n^{(2)}},
$$
and the map $\rho_*\colon E\ot\ov{E}^{\ot^*}\ot\to \bigl(E\ot\ov{E}^{\ot^*}\ot\bigr)\ot_k
\breve H$ is a morphism of mixed complexes. This last fact implies that if $C$ is a
subcoalgebra of $\breve H$, then
$$
b\bigl(E\ot\ov{E}^{\ot^n}\ot^C\bigr) \subseteq E\ot\ov{E}^{\ot^{n-1}}\ot^C\quad\text{and}\quad
B\bigl(E\ot\ov{E}^{\ot^n}\ot^C\bigr)\subseteq E\ot\ov{E}^{\ot^{n+1}}\ot^C.
$$
Let $\bigl(E\ot\ov{E}^{\ot^*}\ot^C,b^C,B^C\bigr)$ be the mixed subcomplex of
$\bigl(E\ot\ov{E}^{\ot^*}\ot,b,B\bigr)$, with modules $E\ot\ov{E}^{\ot^n}\ot^C$. We let
$\HH_*^{K,C}(E)$, $\HC_*^{K,C}(E)$, $\HP_*^{K,C}(E)$ and $\HN_*^{K,C}(E)$ denote its
Hochschild, cyclic, periodic and negative homology, respectively.

\smallskip

Similarly, for each $n\ge 0$, the module $\wh{X}_n$ is an $\breve H$-comodule via
$$
\rho_n\bigl(\bigl[a_0\gamma(h_0)\ot_A \ov{\gamma(\bh_{1s})}\ot\ba_{1,n-s}\bigr]\bigr)\! = \!
\Bigl[a_0\gamma(h_0^{(1)})\ot_A \ov{\gamma(\bh_{1s}^{(1)} )} \ot \ba_{1,n-s}\Bigr]\! \ot \!
\ov{h_0^{(2)}\dots h_s^{(2)}},
$$
and the map $\rho_*\colon \wh{X}_* \to \wh{X}_*\ot \breve H$ is a morphism of mixed complexes.
Consequently, if $C$ is a subcoalgebra of $\breve H$, then
$$
\wh{d}_n(\wh{X}_n^C) \subseteq\wh{X}_{n-1}^C\quad\text{and}\quad \wh{D}_n(\wh{X}_n^C) \subseteq
\wh{X}_{n+1}^C.
$$
Let $\bigl(\wh{X}^C,\wh{d}^C,\wh{D}^C\bigr)$ be the mixed subcomplex of
$\bigl(\wh{X},\wh{d},\wh{D}\bigr)$ with modules $\wh{X}_n^C$. The homotopy equivalent data
introduced in Theorem~2.4 induces by restriction a homotopy equivalent data between
$\bigl(\wh{X}^C,\wh{d}^C,\wh{D}^C\bigr)$ and $\bigl(E\ot\ov{E}^{\ot^*}\ot^C,b^C,B^C\bigr)$. So,
$\HH_*^{K,C}(E)$, $\HC_*^{K,C}(E)$, $\HP_*^{K,C}(E)$ and $\HN_*^{K,C}(E)$ are the Hochschild,
cyclic, periodic and negative homology of $\bigl(\wh{X}^C,\wh{d}^C,\wh{D}^C\bigr)$,
respectively.

\smallskip

Suppose now the cocycle $f$ is invertible. A direct computation shows that the $\breve
H$-coaction of $(\ov{X},\ov{d},\ov{D})$, obtained by transporting the one of
$(\wh{X},\wh{d},\wh{D})$ through $\theta\colon(\wh{X},\wh{d},\wh{D})\to
(\ov{X},\ov{d},\ov{D})$, is given by
$$
[a_0\gamma(h_0)\ot\ba_{1r}]\!\ot_k\! \bh_{1s} \mapsto \bigl[a_0h_0^{(1)}\ot\ba_{1r}\bigr]
\!\ot_k\!\bh_{1s}^{(2)} \ot h_0^{(2)}S(h_1^{(1)}\cdots h_s^{(1)})h_1^{(3)}\cdots h_s^{(3)}.
$$
This implies that if if $\breve H$ is cocommutative, then
$$
\ov{X}_n^C = \bigoplus_{r+s = n} \ov{X}_{rs}^C = \bigoplus_{r+s = n} E^C\ot \ov{A}^{\ot^r} \ot
\ov{H}^{\ot^s}.
$$
For each subcoalgebra $C$ of $\breve H$, we consider the mixed subcomplex $\bigl(\ov{X}^C,
\ov{d}^C,\ov{D}^C\bigr)$ of $\bigl(\ov{X},\ov{d},\ov{D}\bigr)$ with modules $\ov{X}^C_n$. It is
clear that $\theta$ induces an isomorphism
$$
\theta^C\colon \bigl(\wh{X}^C,\wh{d}^C,\wh{D}^C\bigr)\to \bigl(\ov{X}^C,\ov{d}^C,\ov{D}^C).
$$
So, $\HH_*^{K,C}(E)$, $\HC_*^{K,C}(E)$, $\HP_*^{K,C}(E)$ and $\HN_*^{K,C}(E)$ are the
Hochschild, cyclic, periodic and negative homology of $\bigl(\ov{X}^C,\ov{d}^C,\ov{D}^C\bigr)$,
respectively.

\smallskip

By the discussion at the beginning of this subsection, if $\breve H$ decomposes as a direct sum
of a family $(C_i)_{i\in I}$ of subcoalgebras, then
\begin{align*}
& \bigl(E\ot\ov{E}^{\ot^*}\ot,b,B\bigr) = \bigoplus_{i\in I}
\bigl(E\ot\ov{E}^{\ot^*}\ot^{C_i},b^{C_i},B^{C_i}\bigr)\\
& \bigl(\wh{X},\wh{d},\wh{D}\bigr) = \bigoplus_{i\in I} \bigl(\wh{X}^{C_i},\wh{d}^{C_i},
\wh{D}^{C_i}\bigr)
\intertext{and}
& \bigl(\ov{X},\ov{d},\ov{D}\bigr) = \bigoplus_{i\in I} \bigl(\ov{X}^{C_i},\ov{d}^{C_i},
\ov{D}^{C_i}\bigr).
\end{align*}
In particular $\HH^K_*(E) = \bigoplus_{i\in I} \HH_*^{K,C_i}(E)$, etcetera.

\smallskip

In the sequel we use the notations introduced in Subsection~3.1 and 3.2.

\begin{lemma} Assume that $\breve{H}$ is cocommutative and $\ov{H}$ is a flat
$k$-module. If $C$ is a subcoalgebra of $\breve{H}$, then for each $r,s\ge 0$,
\begin{align*}
&\wt{d}\bigl(\Ho^K_r(A,E^C)\ot_k \ov{H}^{\ot^n}\bigr)\subseteq \Ho^K_r(A,E^C)\ot_k
\ov{H}^{\ot^{n-1}}
\intertext{and}
&\wt{D}\bigl(\Ho^K_r(A,E^C)\ot_k \ov{H}^{\ot^n}\bigr)\subseteq \Ho^K_r(A,E^C)\ot_k
\ov{H}^{\ot^{n+1}}.
\end{align*}
\end{lemma}

\begin{proof} Left to the reader.
\end{proof}

\begin{proposition} Assume that $\breve{H}$ is cocommutative and $\ov{H}$ is a flat
$k$-module. Let $C$ be a subcoalgebra of $\breve{H}$ and let
$$
\wt{\Ho^K_r(A,E^C)} = \bigl(\Ho^K_r(A,E^C)\ot_k \ov{H}^{\ot^*},\wt{d}_*^C,\wt{D}_*^C\bigr)
$$
be the submixed complex of $\wt{\Ho^K_r(A,E)}$ with modules $\Ho^K_r(A,E^C)\ot_k
\ov{H}^{\ot^n}$. There is a convergent spectral sequence
$$
E^2_{sr} = \HC_s\bigl(\wt{\Ho^K_r(A,E^C)}\bigr) \Rightarrow \HC^{K,C}_{r+s}(E).
$$
\end{proposition}

\begin{proof} Left to the reader.
\end{proof}

\begin{lemma} Assume that $\breve{H}$ is cocommutative. If $C$ is a subcoalgebra of
$\breve{H}$, then $\mathcal{M}_n^C = E^C\ot\ov{A}^{\ot^n}\ot$ is an $H$-submodule of
$\mathcal{M}_n$ for each $n\ge 0$. Moreover
$$
\partial\bigl(\Ho_s(H,\mathcal{M}_n)\bigr)\subseteq \Ho_s(H,\mathcal{M}_{n-1}) \quad \text{and}\quad
\mathcal{D}\bigl(\Ho_s(H,\mathcal{M}_n)\bigr)\subseteq \Ho_s(H,\mathcal{M}_{n+1}).
$$
\end{lemma}

\begin{proof} Left to the reader.
\end{proof}

\begin{proposition} Assume that $\breve{H}$ is cocommutative. Let $C$ be a subcoalgebra of
$\breve{H}$ and let
$$
\wt{\Ho^K_s(H,E^C)} = \bigl(\Ho_s(H,\mathcal{M}^C_*),\partial_*,\mathcal{D}_*\bigr)
$$
be the submixed complex of $\wt{\Ho^K_s(H,E)}$ with modules $\Ho_s(H,\mathcal{M}^C_n)$. There
is a convergent spectral sequence
$$
\mathcal{E}^2_{rs} = \HC_r\bigl(\wt{\Ho^K_s(H,E^C)}\bigr) \Rightarrow \HC^{K,C}_{r+s}(E).
$$
\end{proposition}

\begin{proof} Left to the reader.
\end{proof}

\appendix
\section{}
This appendix is devoted to prove Propositions~\ref{prop2.1}, \ref{prop2.2} and~\ref{prop2.5}.

\begin{lemma}\label{leA.1} We have
$$
\ov{\sigma}_{n+1} = -  \sigma_{0,n+1}^0 \xcirc \sigma_{n+1}^{-1} \xcirc \mu_n + \sum_{r=0}^n
\sum_{l=0}^{n-r} \sigma_{r+l+1,n-r-l}^l,
$$
\end{lemma}

\begin{proof} By the definition of $\mu$, $\sigma^{-1}$ and $\ov{\sigma}$ it suffices to prove
that
$$
\sigma^l(E\ot_A (E/A)^{\ot_{\!A}^{n+1}} \ot_A A) = 0\quad\text{for all $l\ge 1$.}
$$
Assume the result is false and let $l\ge 1$ be the minimal upper index for which the above
equality is wrong. Let $\bx\in E\ot_A (E/A)^{\ot_{\!A}^{n+1}} \ot_A A$. Then
$$
\sigma^l(\bx) = - \sum_{i=0}^{l-1} \sigma^0 \xcirc d^{l-i} \xcirc\sigma^i(\bx) = - \sigma^0
\xcirc d^l \xcirc\sigma^0(\bx).
$$
But, because $\sigma^0(\bx)\in E\ot_A (E/A)^{\ot_{\!A}^{n+1}}\ot K$, from the definition of
$d^l$ it follows that $d^l \xcirc\sigma^0(\bx)\in \ima(\si_0)$. Since $\sigma^0\xcirc\sigma^0 =
0$, this implies that $\sigma^l(\bx) = 0$, which contradicts the assumption.
\end{proof}

\begin{lemma}\label{leA.2} The contracting homotopy $\ov{\sigma}$ satisfies $\ov{\sigma}\xcirc
\ov{\sigma} = 0$.
\end{lemma}

\begin{proof} By Lemma~\ref{leA.1} it will be sufficient to see that $\sigma^0 \xcirc
\sigma^{-1}\xcirc \mu\xcirc \sigma^0\xcirc \sigma^{-1}\xcirc \mu = 0$ and $\sigma^l \xcirc
\sigma^{l'} = 0$ for all $l,l'\ge 0$. The first equality follows from the fact that $\mu\xcirc
\sigma^0 = \ide$ and $\sigma^{-1}\xcirc\sigma^{-1} = 0$. We now prove the last one. An
inductive argument shows that there exists a map $\gamma^l$ such that $\sigma^l =
\sigma^0\xcirc \gamma^l \xcirc\sigma^0$ for all $l\ge 1$. So $\sigma^{l'}\xcirc \sigma^l = 0$,
since clearly $\sigma^0\xcirc\sigma^0 = 0$.
\end{proof}

\begin{remark}\label{reA.3} The previous lemma implies that $\psi_n(\byy\ot 1) =
(-1)^n\ov{\sigma} \xcirc\psi(\byy)$ for all $n\ge 1$.
\end{remark}

Let $L_{rs}\subseteq U_{rs}$ be the $k$-submodules of $E\ot_A (E/A)^{\ot_{\!A}^s} \ot
\ov{A}^{\ot^r}\ot E$ generated by the simple tensors of the form
$$
1\ot_A \ov{\gamma(\bh_{1s})}\ot\ba_{1r}\ot 1\quad \text{and}\quad 1\ot_A \ov{\gamma(\bh_{1s})}
\ot \ba_{1r}\ot \gamma(h),
$$
respectively.

\smallskip

Note that under the identification $X_{rs}\simeq E\ot_k\ov{H}^{\ot_k^s}\ot\ov{A}^{\ot^r}\ot E$,
the subspaces and $L_{rs}$ and $U_{rs}$ of $X_{rs}$ correspond to $k\ot_k\ov{H}^{\ot_k^s}\ot
\ov{A}^{\ot^r}\ot k$ and $k\ot_k\ov{H}^{\ot_k^s}\ot \ov{A}^{\ot^r}\ot \mathcal{H}$,
respectively

\begin{lemma}\label{leA.4} It is true that $d^l(L_{rs}) \subseteq U_{r+l-1,s-l}$, for each
$l\ge 2$. Moreover
\begin{align*}
d^1(L_{rs}) & \subseteq EL_{r,s-1} + U_{r,s-1}
\end{align*}
\end{lemma}

\begin{proof} We proceed by induction on $l$ and $r$. For $l=1$ and $r\ge 0$, the result
follows immediately from Theorem~\ref{Formula d^1}. Assume that $s\ge l>1$, $r=0$ and that the
result for $l\ge 2$ is true for every $d^j_{r's'}$'s with arbitrary $r',s'$ and $j<l$. Let $\bx
= 1\ot_A \ov{\gamma(\bh_{1s})}\ot 1$. By the very definition of $d^l$, the above inclusion of
$d^1(L_{rs})$, and the inductive hypothesis \allowdisplaybreaks
\begin{align*}
d^l(\bx) & = - \sum_{j=1}^{l-1} \sigma^0\xcirc d^{l-j}\xcirc d^j(\bx)\\
&\in\sigma^0\xcirc d^{l-1}(E L_{0,s-1})+ \sum_{j=1}^{l-1}\sigma^0\xcirc d^{l-j} (U_{j-1,s-j})\\
& = \sum_{j=1}^{l-1} \sigma^0\xcirc d^{l-j}(U_{j-1,s-j}),
\end{align*}
where the last equality follows from the fact that
$$
\ima(\sigma^0)\subseteq \ker(\sigma^0)\quad\text{and}\quad d^{l-1}(EL_{0,s-1})\subseteq
\ima(\sigma^0),
$$
by the definition of $d^{l-1}$. Now, by the inductive hypothesis,
\begin{align*}
& d^{l-j}(U_{j-1,s-j})\subseteq L_{l-2,s-l}E\quad\text{for $l-j>1$}\\
\intertext{and}
& d^1(U_{l-2,s-l+1})\subseteq E U_{l-2,s-l} + L_{l-2,s-l}E.
\end{align*}
Thus, by the definition of $\sigma^0$, we have $d^l(\bx)\in U_{l-1,s-l}$. Suppose now that
$r>0$ and the result is true for all the $d^j_{r's'}$'s with arbitrary $r',s'$ and $j<l$, and
for all the $d^l_{r's'}$'s with arbitrary $s'$ and $r'<r$. Let $\bx = 1\ot_A
\ov{\gamma(\bh_{1s})}\ot \ba_{1r}\ot 1$. Arguing as above we see that
$$
d^l(\bx) \equiv - \sigma^0\xcirc d^l\xcirc d^0(\bx) \pmod{U_{r+l-1,s-l}}.
$$
Finally, by the definition of $d^0$ and the inductive hypothesis,
\begin{align*}
\sigma^0\xcirc d^l\xcirc d^0(\bx) &\in \sigma^0\xcirc d^l(AL_{r-1,s} + L_{r-1,s}A)\\
& \subseteq \sigma^0(AU_{r+l-2,s-l} + U_{r+l-2,s-l}A)\\
& \subseteq U_{r+l-1,s-l},
\end{align*}
which finishes the proof.
\end{proof}

\smallskip

We recursively define $\gamma(\bh_{1s}) * \ba_{1r}$ by
\begin{itemize}

\smallskip

\item $\gamma(\bh_{1s}) * \ba_{1r} = \ba_{1r}$ if $s=0$ and $\gamma(\bh_{1s}) * \ba_{1r} =
\gamma(\bh_{1s})$ if $r=0$,

\smallskip

\item If $r,s\ge 1$, then $\gamma(\bh_{1s}) * \ba_{1r} = \sum_{i=0}^r (-1)^i
\gamma(\bh_{1,s-1}) * \ba_{1i}^{h_s^{(1)}}\ot \gamma(h_s^{(2)})\ot \ba_{i+1,r}$.

\smallskip

\end{itemize}

Let $V_n$ be the $k$-submodule of $B_n(E)$ generated by the simple tensors $1\ot \bx_{1n}\ot 1$
such that $x_i\in A\cup \mathcal{H}$ for $1\le i\le n$.

\smallskip

Recall that $H\!\cdot\!\ima(f)$ denotes the minimal $k$-submodule of $A$ that includes
$\ima(f)$ and is closed under the weak action of $H$. We will denote by $C_n$ the
$E$-subbimodule of $E\ot \ov{E}^{\ot^n}\ot E$ generated by all the simple tensors $1\ot
x_1\ot\cdots\ot x_n\ot 1$ with some $x_i$ in $H\!\cdot\!\ima(f)$.

\begin{proposition}\label{propA.5} The map $\phi$ satisfies
$$
\phi(1\ot_A\ov{\gamma(\bh_{1i})}\ot\ba_{1,n-i}\ot 1) \equiv 1\ot\gamma(\bh_{1i})\!*
\ba_{1,n-i}\ot 1
$$
module $F^{i-1}(B_n(E))\cap V_n\cap C_n$.

\end{proposition}

\begin{proof} We proceed by induction on $n$. Let $\bx = 1\ot_A \ov{\gamma(\bh_{1i})}\ot
\ba_{1,n-i}\ot 1$. By item~(2) of Theorem~\ref{Formula d^1}, the fact that $d^l(\bx) \in
U_{n-i+l-1,i-l}$ (by Lemma~\ref{leA.4}), and the inductive hypothesis
$$
\xi\xcirc\phi\xcirc d^l(\bx)\in F^{i-l+1}(B_n(E))\cap V_n\cap C_n \qquad\text{for all $l>1$,}
$$
So,
$$
\phi(\bx) \equiv \xi\xcirc\phi\xcirc d^0(\bx) + \xi\xcirc\phi\xcirc d^1(\bx)
\pmod{F^{i-1}(B_n(E))\cap V_n\cap C_n}.
$$
Moreover, by the definition of $d^0$ and Theorem~\ref{Formula d^1}
\begin{align*}
&\xi\xcirc\phi\xcirc d^0(\bx) = (-1)^n \xi\xcirc\phi(1\ot_A\ov{\gamma(\bh_{1i})}\ot
\ba_{1,n-i}),
\intertext{and}
&\xi\xcirc\phi\xcirc d^1(\bx) = (-1)^i \xi\xcirc \phi(1\ot_A\ov{\gamma(\bh_{1,i-1})}\ot
\ba_{1,n-i}^{h_i^{(1)}}\ot \gamma(h_i^{(2)})),
\end{align*}
since $\phi(EL_{n-s-1,s})\subseteq E\ot \ov{E}^{\ot^{n-1}}\ot K\subseteq \ker(\xi)$. The proof
can be now easily finished using the inductive hypothesis.
\end{proof}

In the sequel we let $J_n$ denote the $E$-subbimodule of $X_n$ generated by all the simple
tensors
$$
1\ot_A x_1\ot_A\cdots\ot_A x_s\ot a_1\ot\cdots\ot a_r \ot 1\qquad \text{($r+s = n$),}
$$
with some $a_i$ in the image of the cocycle $f$.

\begin{lemma}\label{leA.6} We have:

\begin{enumerate}

\smallskip

\item Let $\bx = 1\ot_A \ov{\gamma(\bh_{1i})}\ot \ba_{i+1,n}$. If $i<n$, then
$$
\ov{\sigma}(\bx) = \sigma^0(\bx) = (-1)^n \ot_A \ov{\gamma(\bh_{1i})}\ot \ba_{i+1,n}\ot 1.
$$

\smallskip

\item If $\bz = 1\ot_A \ov{\gamma(\bh_{1,i-1})}\ot\ba_{i,n-1}\ot a_n\gamma(h_n)$, then
$\sigma^l(\bz) \in U_{n-i+l+1,i-1-l}$ for $l\ge 0$ and $\sigma^l(\bz) \in J_n$ for $l\ge 1$.

\smallskip

\item If $\bz = 1\ot_A \ov{\gamma(\bh_{1,i-1})}\ot\ba_{i,n-1}\ot\gamma(h_n)$, then
$\sigma^l(\bz) = 0$ for $l\ge 0$.

\smallskip

\item If $\bz = 1\ot_A  \ov{\gamma(\bh_{1,i-1})}\ot\ba_{i,n-1}\ot a_n\gamma(h_n)$ and $i<n$,
then $\ov{\sigma}(\bz) \equiv \sigma^0(\bz)$, module $\bigoplus_{l=0}^{i-2} (U_{n-l,l}\cap
J_n)$.

\smallskip

\item If $\byy = 1\ot_A  \ov{\gamma(\bh_{1,n-1})}\ot a_n\gamma(h_n)$, then $\ov{\sigma}(\byy)
\equiv - \sigma^0\xcirc\sigma^{-1}\xcirc\mu(\byy) + \sigma^0(\byy)$, module
$\bigoplus_{l=0}^{n-2} (U_{n-l,l}\cap J_n)$.

\smallskip

\item If $\bz = 1\ot_A  \ov{\gamma(\bh_{1,n-1})}\ot \gamma(h_n)$, then $\ov{\sigma}(\bz) = -
\sigma^0\xcirc\sigma^{-1}\xcirc\mu(\bz)$.

\smallskip

\item If $\bz = 1\ot_A  \ov{\gamma(\bh_{1,i-1})}\ot\ba_{i,n-1}\ot \gamma(h_n)$ and $i<n$, then
$\ov{\sigma}(\bz) = 0$.

\end{enumerate}

\end{lemma}

\begin{proof} The first assertion improves item~(b) of the proof of \cite[Proposition
1.2.2]{G-G}. We first claim that if $l\ge 1$, then $\sigma^l(\bx) = 0$. We proceed by induction
on $l$. By the recursive definition of $\sigma^l$ and the inductive hypothesis
$$
\sigma^l(\bx) = -\sum_{i=0}^{l-1} \sigma^0\xcirc d^{l-i}\xcirc \sigma^i(\bx)  = -
\sigma^0\xcirc d^l\xcirc \sigma^0(\bx) = (-1)^{n-1} \sigma^0\xcirc d^l (\bx\ot 1).
$$
In order to finish the proof of the claim it suffices to note that $\sigma^0\xcirc \sigma^0 =
0$ and that, by the very definition, $d^l(\bx\ot 1)\in \ima(\sigma^0)$. When $i<n-1$ item~(1)
follows clearly from the claim. When $i=n-1$ it is necessary to see also that $\sigma^l\xcirc
\sigma^{-1}\xcirc \mu(\bx) = 0$, which is immediate, since $\sigma^{-1}\xcirc \mu(\bx) = 0$ by
the definitions of $\mu$ and $\sigma^{-1}$. We next prove the first part of item~(2). By
definition this is clear for $\sigma^0$. Assume the result is valid for $\sigma^i$ with $i<l$.
Then, by Lemma~\ref{leA.4},
\begin{align*}
\sigma^l(\bz) & = - \sum_{j=0}^{l-1} \sigma^0 \xcirc d^{l-j}\xcirc \sigma^j(\bz)\\
& \subseteq \sum_{j=0}^{l-1} \sigma^0 \xcirc d^{l-j}(U_{n-i+j+1,i-1-j})\\
& \subseteq \sigma^0(EU_{n-i+l,i-1-l})+ \sigma^0(U_{n-i+l,i-1-l}E)\\
& = U_{n-i+l+1,i-1-l},
\end{align*}
as desired. We now prove the second part. By Theorem~\ref{Formula d^1}, the recursive
definition of $\sigma^l$ and the definition of $\sigma^0$, we know that
$$
\sigma^l(\bz) = - \sum_{j=0}^{l-1} \sigma^0 \xcirc d^{l-j}\xcirc \sigma^j(\bz)\equiv  -
\sigma^0 \xcirc d^1\xcirc \sigma^{l-1}(\bz)\pmod{J_n}.
$$
Since $\sigma^0 \xcirc d^1\xcirc \sigma^{l-1}(\bz)\in \sigma^0 \xcirc d^1(U_{n-i+l,i-l})$, in
order to finish the proof it suffices to see that $\sigma^0 \xcirc d^1(U_{n-i+l,i-l})\subseteq
J_n$, which is a direct consequence of Theorem~\ref{Formula d^1} and the definition of
$\sigma^0$. Item~(3) follows immediately by induction on $l$. Items~(4) and~(5) follow easily
from the definition of $\ov{\sigma}$, item~(2) and Lemma~\ref{leA.1}. Finally, items~(6)
and~(7) follow from the definition of $\ov{\sigma}$, item~(3) and Lemma~\ref{leA.1}.
\end{proof}

Let $V'_n$ be the $k$-submodule of $E\ot \ov{E}^{\ot^n}\ot E$ generated by the simple tensors
$1\ot\bx_{1n}\ot 1$ such that $\#(\{j:x_j\notin A\cup \mathcal{H}\})\le 1$ (Note that $V_n
\subseteq V'_n$).

\begin{proposition}\label{propA.7} Let $R_i = F^i(B_n(E))\setminus F^{i-1}(B_n(E))$. The
following equalities hold:

\begin{enumerate}

\smallskip

\item $\psi(1\ot\gamma(\bh_{1i})\ot\ba_{i+1,n}\ot 1) = 1\ot_A\ov{\gamma(\bh_{1i})}\ot
\ba_{i+1,n} \ot 1$.

\smallskip

\item If $\bx = 1\ot \bx_{1n}\ot 1\in R_i\cap V_n$ and there exists $1\le j\le i$ such that
$x_j\in A$, then $\psi(\bx) = 0$.

\smallskip

\item If $\bx = 1\ot \gamma(\bh_{1,i-1})\ot a_i\gamma(h_i)\ot\ba_{i+1,n} \ot 1$, then
\begin{align*}
\qquad\qquad \psi(\bx) & \equiv 1\ot_A \ov{\gamma(\bh_{1,i-1})}\ot_A a_i\gamma(h_i)
\ot\ba_{i+1,n} \ot 1\\
&+ 1\ot_A \ov{\gamma(\bh_{1,i-1})}\ot a_i\ot \ba_{i+1,n}^{h_i^{(1)}}\ot \gamma(h_i^{(2)}),
\end{align*}
module $\bigoplus_{l=0}^{i-2} (U_{n-l,l}\cap J_n)$.

\smallskip

\item If $\bx = 1\ot \gamma(\bh_{1,j-1})\ot a_j\gamma(h_j)\ot\gamma(\bh_{j+1,i})\ot\ba_{i+1,n}
\ot 1$ with $j<i$, then
$$
\qquad\qquad \psi(\bx)\equiv 1\ot_A \ov{\gamma(\bh_{1,j-1})}\ot_A a_j\gamma(h_j)\ot_A
\ov{\gamma(\bh_{j+1,i})} \ot \ba_{i+1,n} \ot 1,
$$
module $\bigoplus_{l=0}^{i-2} (U_{n-l,l}\cap J_n)$.

\smallskip

\item If $\bx = 1\ot \gamma(\bh_{1,i-1}) \ot \ba_{i,j-1}\ot a_j\gamma(h_j) \ot\ba_{j+1,n} \ot 1$
with $j>i$, then
$$
\qquad\qquad\psi(\bx)\equiv 1\ot_A \ov{\gamma(\bh_{1,i-1})}\ot \ba_{ij}\ot
\ba_{j+1,n}^{h_j^{(1)}} \ot\gamma(h_j^{(2)}),
$$
module $\bigoplus_{l=0}^{i-2} (U_{n-l,l}\cap J_n)$.

\smallskip

\item If $\bx = 1\ot \bx_{1n}\ot 1\in R_i\cap V'_n$ and there exists $1\le j_1 < j_2\le i$
such that $x_{j_1}\in A$ and $x_{j_2}\in \mathcal{H}$, then $\psi(\bx) = 0$.

\end{enumerate}

\end{proposition}

\begin{proof} 1) We proceed by induction on $n$. The case $n=0$ is trivial. Suppose $n>0$ and
the result is valid for $n-1$. Assume first that $i<n$. By Remark~\ref{reA.3} and the inductive
hypothesis,
\begin{align*}
\psi(1\ot\gamma(\bh_{1i})\ot\ba_{i+1,n}\ot 1) & = (-1)^n \ov{\sigma}\xcirc \psi(1\ot
\gamma(\bh_{1i}) \ot\ba_{i+1,n})\\
& = (-1)^n \ov{\sigma}(1\ot_A\ov{\gamma(\bh_{1i})}\ot\ba_{i+1,n}),
\end{align*}
and the result follows from item~(1) of Lemma~\ref{leA.6}. Assume now that $i=n$. By
Remark~\ref{reA.3}, the inductive hypothesis and item~(6) of Lemma~\ref{leA.6},
\begin{align*}
\psi(1\ot\gamma(\bh_{1n})\ot 1) &= (-1)^n \ov{\sigma}\xcirc \psi(1\ot\gamma(\bh_{1n}))\\
& = (-1)^{n+1} \sigma^0\xcirc \sigma^{-1}\xcirc \mu(1\ot_A\ov{\gamma(\bh_{1,n-1})}\ot
\gamma(h_n)).
\end{align*}
The result follows now immediately from the definitions of $\mu$, $\sigma^{-1}$ and $\sigma^0$.

\smallskip

\noindent 2) We proceed by induction on $n$. Assume first that there exist $j_1<j_2<n$ such
that $x_{j_1}\in A$ and $x_{j_2}\in \mathcal{H}$. By Remark~\ref{reA.3} and the inductive
hypothesis,
$$
\psi(\bx)= (-1)^n \ov{\sigma}\xcirc \psi(1\ot\bx_{1n}) = (-1)^n \ov{\sigma}(0) = 0.
$$
Assume now that $\bx_{1n} = \gamma(\bh_{1,i-1})\ot \ba_{i,n-1}\ot \gamma(h_n)$. By
Remark~\ref{reA.3} and item~(1),
$$
\psi(\bx) = (-1)^n \ov{\sigma}\xcirc \psi(1\ot\bx_{1n}) = (-1)^n \ov{\sigma}(1\ot_A
\ov{\gamma(\bh_{1,i-1})}\ot \ba_{i,n-1}\ot \gamma(h_n)),
$$
and the result follows from item~(7) of Lemma~\ref{leA.6}.

\smallskip

\noindent 3) We proceed by induction on $n$. Assume first that $i<n$. Let
\begin{align*}
\byy & = 1\ot_A \ov{\gamma(\bh_{1,i-1})}\ot_A a_i\gamma(h_i)\ot\ba_{i+1,n},\\
\bz & = 1\ot_A \ov{\gamma(\bh_{1,i-1})}\ot a_i\ot \ba_{i+1,n-1}^{h_i^{(1)}} \ot
\gamma(h_i^{(2)})a_n.
\end{align*}
By Remark~\ref{reA.3} and the inductive hypothesis,
$$
\psi(\bx) = (-1)^n \ov{\sigma}\xcirc \psi(1\ot \gamma(\bh_{1,i-1})\ot a_i\gamma(h_i)
\ot\ba_{i+1,n}) \equiv (-1)^n \ov{\sigma}(\byy+\bz),
$$
module $\ov{\sigma}\bigl(\bigoplus_{l=0}^{i-2} (U_{n-1-l,l}\cap J_{n-1})A\bigr)$. So, by
items~(1) and~(4) of Lemma~\ref{leA.6},
$$
\psi(\bx) \equiv (-1)^n \sigma^0(\byy+\bz),
$$
module $\bigoplus_{l=0}^{i-2} (U_{n-l,l}\cap J_n)+\sigma^0 \bigl(\bigoplus_{l=0}^{i-2}
(U_{n-1-l,l}\cap J_{n-1})A\bigr)$. Using the definition of $\sigma^0$ we obtain immediately the
desired expression for $\psi(\bx)$. Assume now that $i=n$. Let
$$
\byy   = 1\ot\gamma(\bh_{1,n-1})\ot a_n\gamma(h_n)\quad\text{and}\quad \bz  = 1\ot_A
\ov{\gamma(\bh_{1,n-1})}\ot a_n\gamma(h_n).
$$
By Remark~\ref{reA.3}, item~(1) of the present proposition and item~(5) of Lemma~\ref{leA.6},
$$
\psi(\bx)= (-1)^n \ov{\sigma}\xcirc \psi(\byy) = (-1)^n \ov{\sigma}(\bz)\equiv (-1)^{n+1}
\sigma^0\xcirc \sigma^{-1}\xcirc \mu(\bz) + (-1)^n \sigma^0(\bz) ,
$$
module $\bigoplus_{l=0}^{n-2} (U_{n-l,l}\cap J_n)$. The established formula for $\psi(\bx)$
follows now easily from the definitions of $\mu$, $\sigma^{-1}$ and $\sigma^0$.

\smallskip

\noindent 4) We proceed by induction on $n$. When $i< n$ the same argument that in item~(3)
works. Assume now that $j<i-1$ and $i=n$. Let
\begin{align*}
& \byy  = 1\ot \gamma(\bh_{1,j-1})\ot a_j\gamma(h_j)\ot\gamma(\bh_{j+1,n}),\\
&\bz  = 1\ot_A \ov{\gamma(\bh_{1,j-1})}\ot_A a_j\gamma(h_j)\ot_A\ov{\gamma(\bh_{j+1,n-1})}\ot
\gamma(h_n).
\end{align*}
By Remark~\ref{reA.3} and the inductive hypothesis,
$$
\psi(\bx)= (-1)^n \ov{\sigma}\xcirc \psi(\byy) \equiv (-1)^n \ov{\sigma}(\bz) ,
$$
module $\ov{\sigma}\bigl(\bigoplus_{l=0}^{n-3} (U_{n-1-l,l}\cap J_{n-1})E\bigr)$. So, by
items~(4) and~(6) of Lemma~\ref{leA.6},
$$
\psi(\bx)\equiv (-1)^{n+1} \sigma^0\xcirc \sigma^{-1}\xcirc \mu(\bz),
$$
module $\bigoplus_{l=0}^{n-4} (U_{n-l,l}\cap J_n)+\sigma^0\bigl(\bigoplus_{l=0}^{n-3}
(U_{n-1-l,l}\cap J_{n-1})E\bigr)$. The formula for $\psi(\bx)$ follows now easily from the
definitions of $\mu$, $\sigma^{-1}$ and $\sigma^0$. Assume finally that $j = i-1$ and $i=n$.
Let
\begin{align*}
& \byy  = 1\ot_A \ov{\gamma(\bh_{1,n-2})}\ot_A a_{n-1}\gamma(h_{n-1})\ot\gamma(h_n),\\
&\bz = 1\ot_A \ov{\gamma(\bh_{1,n-2})}\ot a_{n-1}\ot\gamma(h_{n-1}) \gamma(h_n).
\end{align*}
By Remark~\ref{reA.3} and item~(3),
$$
\psi(\bx)= (-1)^n \ov{\sigma}\xcirc \psi(1\ot \gamma(\bh_{1,n-2})\ot a_{n-1}\gamma(h_{n-1})
\ot\gamma(h_n)) \equiv (-1)^n \ov{\sigma}(\byy+\bz) ,
$$
module $\ov{\sigma}\bigl(\bigoplus_{l=0}^{n-3} (U_{n-1-l,l}\cap J_{n-1})E\bigr)$. So, by the
fact that $\sigma^0(\bz)\in U_{2,n-2}\cap J_n$, and items~(4) and~(6) of Lemma~\ref{leA.6},
$$
\psi(\bx)\equiv (-1)^{n+1} \sigma^0\xcirc \sigma^{-1}\xcirc \mu(\byy),
$$
module $\bigoplus_{l=0}^{n-2} (U_{n-l,l}\cap J_n)+\sigma^0\bigl(\bigoplus_{l=0}^{n-3}
(U_{n-1-l,l}\cap J_{n-1})E\bigr)$. The formula for $\psi(\bx)$ follows now easily from the
definitions of $\mu$, $\sigma^{-1}$ and $\sigma^0$.

\smallskip

\noindent 5) We proceed by induction on $n$. Let
\begin{align*}
& \byy = 1\ot \gamma(\bh_{1,i-1}) \ot \ba_{i,j-1}\ot a_j\gamma(h_j) \ot \ba_{j+1,n},\\
& \bz = 1\ot_A \ov{\gamma(\bh_{1,i-1})}\ot\ba_{ij}\ot \ba_{j+1,n-1}^{h_j^{(1)}}
\ot\gamma(h_j^{(2)})a_n.
\end{align*}
By Remark~\ref{reA.3} and item~(1) or the inductive hypothesis (depending on $j=n$ or $j<n$),
$$
\psi(\bx)= (-1)^n \ov{\sigma}\xcirc \psi(\byy) \equiv (-1)^n \ov{\sigma}(\bz),
$$
module $\ov{\sigma}\bigl(\bigoplus_{l=0}^{i-2} (U_{n-l-1,l}\cap J_{n-1})A\bigr)$. Thus, by
item~(4) of Lemma~\ref{leA.6},
$$
\psi(\bx)= (-1)^n \sigma^0\xcirc \psi(\byy) \equiv (-1)^n\sigma^0(\bz),
$$
module $\bigoplus_{l=0}^{i-2} (U_{n-l,l}\cap J_n) + \sigma^0\bigl(\bigoplus_{l=0}^{i-2}
(U_{n-l-1,l}\cap J_{n-1})A\bigr)$. The result is obtained now immediately using the definition
of $\sigma^0$.

\smallskip

\noindent 6) We proceed by induction on $n$. By Remark~\ref{reA.3} and item~(2) or the
inductive hypothesis (depending on $x_n\notin A\cup \mathcal{H}$ or $x_n\in A\cup
\mathcal{H}$),
$$
\psi(\bx)= (-1)^n \ov{\sigma}\xcirc \psi(1\ot\bx_{1n}) = (-1)^n \ov{\sigma}(0) = 0,
$$
as desired.
\end{proof}

\begin{lemma}\label{leA.8} Let $R_i = F^i(B_n(E)) \setminus F^{i-1}(B_n(E))$. The following
equalities hold:

\begin{enumerate}

\smallskip

\item $\phi\xcirc \psi\bigl(1\ot\gamma(\bh_{1i})\ot \ba_{1,n-i}\ot 1\bigr) \equiv
1\ot\gamma(\bh_{1i})*\ba_{1,n-i}\ot 1$ module $F^{i-1}(B_n(E))\cap V_n$.

\smallskip

\item If $\bx = 1\ot\bx_{1n}\ot 1\in R_i\cap V_n$ and there exists $1\le j\le i$ such that
$x_i\in A$, then $\phi\xcirc \psi(\bx) = 0$.

\smallskip

\item If $\bx = 1\ot \gamma(\bh_{1,i-1})\ot a_i\gamma(h_i)\ot\ba_{i+1,n} \ot 1$, then
\begin{align*}
\qquad\qquad \phi\xcirc\psi(\bx) & \equiv a_i^{\bh_{1,i-1}^{(1)}}\ot \Bigl(
\gamma(\bh_{1,i-1}^{(2)}) \ot\gamma(h_i)\Bigr)*\ba_{i+1,n} \ot 1 \\
&+ 1\ot\gamma(\bh_{1,i-1})*\Bigl(a_i\ot \ba_{i+1,n}^{h_i^{(1)}}\Bigr)\ot \gamma(h_i^{(2)}),
\end{align*}
module $F^{i-1}(B_n(E))\cap AV_n + F^{i-2}(B_n(E))\cap V_n\mathcal{H}$.

\smallskip

\item If $\bx = 1\ot \gamma(\bh_{1,j-1})\ot a_j\gamma(h_j)\ot\gamma(\bh_{j+1,i})\ot\ba_{i+1,n}
\ot 1$ with $j<i$, then
$$
\qquad\qquad \phi\xcirc\psi(\bx)\equiv a_j^{\bh_{1,j-1}^{(1)}}\ot
\Bigl(\gamma(\bh_{1,j-1}^{(2)}) \ot \gamma(\bh_{ji})\Bigr)* \ba_{i+1,n}\ot 1,
$$
module $F^{i-1}(B_n(E))\cap AV_n + F^{i-2}(B_n(E))\cap V_n\mathcal{H}$.

\smallskip

\item If $\bx = 1\ot \bx_{1n}\ot 1\in R_i\cap V'_n$ and there exists $1\le j\le i$ such that
$x_j\in A$, then $\phi\xcirc \psi(\bx)\in F^{i-1}(B_n(E))\cap V_n\mathcal{H}$.

\end{enumerate}

\end{lemma}

\begin{proof} Item~(1) follows from item~(1) of Proposition~\ref{propA.7} and
Proposition~\ref{propA.5}, and item~(2) follows from item~(2) of Proposition~\ref{propA.7}. We
next prove item~(3). By item~(3) of Proposition~\ref{propA.7},
\begin{align*}
\phi\xcirc \psi(\bx) & \equiv \phi\bigl(a_i^{\bh_{1,i-1}^{(1)}}\ot_A
\ov{\gamma(\bh_{1,i-1}^{(2)})}\ot_A \gamma(h_i)\ot\ba_{i+1,n} \ot 1\bigr)\\
& + \phi\bigl(1\ot_A \ov{\gamma(\bh_{1,i-1})}\ot a_i\ot \ba_{i+1,n}^{h_i^{(1)}}\ot
\gamma(h_i^{(2)}) \bigr),
\end{align*}
module $\phi\bigl(\bigoplus_{l=0}^{i-2} U_{n-l,l}\bigr)$. So, by Proposition~\ref{propA.5}
\begin{align*}
\phi\xcirc \psi(\bx) & \equiv a_i^{\bh_{1,i-1}^{(1)}}\ot \Bigl(\gamma(\bh_{1,i-1}^{(2)})\ot
\gamma(h_i)\Bigr)*\ba_{i+1,n} \ot 1 \\
& + 1\ot\gamma(\bh_{1,i-1})*\Bigl(a_i\ot \ba_{i+1,n}^{h_i^{(1)}}\Bigr)\ot \gamma(h_i^{(2)}),
\end{align*}
module $F^{i-1}(B_n(E))\cap AV_n + F^{i-2}(B_n(E))\cap V_n\mathcal{H}$. We leave the task to
prove items~(4) and~(5) to the reader.
\end{proof}

\begin{proposition}\label{propA.9} Let $R_i = F^i(B_n(E)) \setminus F^{i-1}(B_n(E))$.  If $\bx
= 1\ot\bx_{1n}\ot 1\in R_i\cap V'_n$, then $\omega(\bx) \in F^i(B_{n+1}(E))\cap V_{n+1}$.
\end{proposition}

\begin{proof} We first claim that if $\bx = 1\ot\bx_{1n}\ot 1\in R_i\cap V_n$, then
$\omega(\bx)=0$. For $n=1$ this is immediate, since $\omega_1 = 0$ by definition. Assume that
$n>1$ and the claim holds for $n-1$. Then,
$$
\omega(\bx) = \xi\bigl(\phi\xcirc \psi(\bx) - (-1)^n \omega(1\ot\bx_{1n})\bigr) = \xi\xcirc
\phi\xcirc \psi(\bx) = 0,
$$
where the last equality follows from the facts that $\phi\xcirc \psi(\bx)\in V_n$ (by items~(1)
and~(2) of Lemma~\ref{leA.8}) and $V_n\subseteq \ker(\xi)$. We now prove the proposition by
induction on $n$. This is trivial for $n=1$ since $w_1 = 0$. Assume that $n>1$ and the
proposition is true for $n-1$. Let $\bx = 1\ot\bx_{1n}\ot 1\in R_i\cap V'_n$. Since
$$
\omega(\bx) = \xi\bigl(\phi\xcirc \psi(\bx) - (-1)^n \omega(1\ot \bx_{1n})\bigr),
$$
and, by items~(3), (4) and~(5) of Lemma~\ref{leA.8},
$$
\xi\xcirc \phi\xcirc \psi(\bx) \in F^i(B_{n+1}(E))\cap V_{n+1},
$$
in order to finish the proof it suffices to check that
\begin{equation*}
\xi\xcirc\omega(1\ot \bx_{1n})\in F^i(B_{n+1}(E))\cap V_{n+1}.
\end{equation*}
By the inductive hypothesis and the claim,

\begin{itemize}

\smallskip

\item If $x_n\in A$, then $\omega(1\ot \bx_{1n})\in F^i(B_n(E))\cap V_nA$,

\smallskip

\item If $x_n\in \mathcal{H}$, then $\omega(1\ot \bx_{1n})\in F^{i-1}(B_n(E))\cap V_n
\mathcal{H}$,

\smallskip

\item If $x_n\notin A\cup \mathcal{H}$, then $\omega(1\ot \bx_{1n})=0$.

\smallskip

\end{itemize}
In all these cases the required inclusion is true.
\end{proof}

\noindent{\bf Proofs of Propositions~\ref{prop2.1}, \ref{prop2.2} and~\ref{prop2.5}.}\enspace
They follow immediately from Propositions~\ref{propA.5}, \ref{propA.9} and~\ref{propA.7},
respectively.\qed

\end{document}